\documentclass[11pt,oneside]{amsart}
\usepackage[utf8]{inputenc}
\usepackage[T1]{fontenc}
\usepackage[english]{babel}
\usepackage[unicode=true,pdfusetitle, bookmarks=true,bookmarksnumbered=false,bookmarksopen=true,bookmarksopenlevel=3, breaklinks=false,pdfborder={0 0 0},backref=false,colorlinks=true]{hyperref}
\hypersetup{linkcolor=blue,anchorcolor=blue,citecolor=blue,urlcolor=black} 
\usepackage{amsfonts}
\usepackage{amssymb}
\usepackage{amsmath}
\usepackage{color}
\usepackage{dsfont}
\usepackage{amsthm}
\usepackage{ytableau}
\usepackage{enumerate}
\usepackage{mathtools}
\usepackage{a4wide}
\usepackage{mathrsfs}
\usepackage{appendix}
\usepackage{bm}
\usepackage{breakurl}
\usepackage{url}

\allowdisplaybreaks[1]

\def\cm{{\mathcal M}}

\renewcommand\hat{\widehat}

\def\r{\right}
\def\lf{\left}

\def\bint{{\ifinner\rlap{\bf\kern.30em--}
\int\else\rlap{\bf\kern.35em--}\int\fi}\ignorespaces}

\def\sbint{{\ifinner\rlap{\bf\kern.32em--}
\hspace{0.078cm}\int\else\rlap{\bf\kern.45em--}\int\fi}\ignorespaces}

\newtheorem{theorem}{Theorem}[section]
\newtheorem{lemma}[theorem]{Lemma}

\newtheorem{Proposition}[theorem]{Proposition}

\theoremstyle{assumption}

\theoremstyle{definition}
\newtheorem{remark}[theorem]{Remark}

\numberwithin{equation}{section}

\numberwithin{equation}{section}

\numberwithin{equation}{section}

\title[]{A noncommutative maximal inequality for Fej\'{e}r means on totally disconnected non-Abelian groups}
\author{Fugui Ding}
\address{Fugui Ding, School of Mathematics and Statistics, Wuhan University, 430072 Wuhan, China.}
\email{fgding163@163.com}

\author{Guixiang Hong}
\address{Guixiang Hong, Institute for Advanced Study in Mathematics, Harbin Institute of Technology, 15000 Harbin, China}
\email{gxhong@hit.edu.cn}

\author{Xumin Wang}
\address{Xumin Wang, Seoul National University, 08826  Seoul, South Korea} 
\email{xumin.wang1124@gmail.com}

\begin{document}
	-
	\begin{abstract}
		In this paper, we explore Fourier analysis for noncommutative $L_p$ space-valued functions on $G$, where $G$ is a totally disconnected non-abelian compact group. By additionally assuming that the value of these functions remains invariant within each conjugacy class,  we establish a noncommutative maximal inequality for Fej\'er means utilizing the associated character system of $G$. This is an operator-valued version of the classical result due to G\'at in \cite{Gat06}. We follow essentially the classical sketch, but due to the noncommutativity, many classical arguments have to be revised. Notably, compared to the classical results. the bounds of our  estimates are explicity calculated.
	\end{abstract}
	
\maketitle

\section{Background and the main result}

Over the past two decades,  there has been significant attention on  totally disconnected groups, especially in non-abelian cases. These groups have found applications in various areas, for instance, in topological group theory, Lie group theory, and operator algebras (see e.g.  \cite{Bur18, CaMo18, CaWi21, IlSp07, Rau19, RiZa10, SpSt13, Laz65, Glo18, BuMo00, BuMo002, Pal01, ShWi13, Wil15, Wil98}). 

In the context of Fourier analysis, our interest lies in the pointwise convergence of the character system of totally disconnected \emph{compact} groups with coefficients in noncommutative $L_p$-spaces throughout this paper. The character system arise canonically from the fact that any totally disconnected compact group can be realized as the complete direct product of its factor structures (see e.g. \cite{Pon66}). 
In other words, a totally disconnected compact group can be constructed as follows. 
Consider a sequence  of number $m:=(m_k)_{k \in \mathbb{N}}$ where  $ m_k \ge 2,k \in \mathbb{N}$  and denote $G_k$ as finite groups of order  $m_k$. Each group $G_k$ is equipped with the discrete topology and the normlized Haar measure $\mu_k$ such that $\mu_k(G_k)=1$, that is, the measure of every singleton of each $G_k$ is ${1}/{m_k}$. The totally disconnected compact group $G$ is composed of the complete direct product of $G_k$ and  equipped with product topology, operation and the corresponding product measure $\mu$. We denote it by 
$$G:=\bigtimes_{k=0}^{\infty}G_k.$$
By the definition of operation, it is evident that $G$ is non-abelian if and only if there exists at least one $G_k$ that is non-ablien. Denote $\rho(G):=\sup_{k \in \mathbb{N}} m_k $. The group $G$ is called bounded if $\rho(G)<\infty$ and unbounded otherwise.

Let $p_k$ denote the number of conjugacy classes in the group $G_k$.  Define
$P_0 :=1,P_{k+1}=p_kP_k$.
Then there is a unique expansion for each natural number $n$,
\begin{equation}\label{eq:unique expansion of n}
	n = \sum_{k=0}^{\infty } n_kP_k, \qquad 0 \le n_k \le p_k-1. 
\end{equation}
Denote by 
$$r_k^0(x_k) = 1,\quad r_k^1(x_k),\cdots , r_k^{p_k-1}(x_k),$$
the irreducible characters of the group $G_k$. 
The domain of $r_k^d$ extends to the whole group $G$ by
$$r_k^d(x):=r_k^d(x_k) , \qquad d \in \{ 0,\cdots,p_k-1\}, \quad G\ni x=(x_k)_{k\geq0}.$$
Now the family of irreducible characters of $G$ can be indexed by natural numbers as following: for any $n\in \mathbb N$, using the unique expansion~\eqref{eq:unique expansion of n}, we define
$$\chi_n(x) = \prod_{k=0}^{\infty} r_k^{n_k}(x).$$
Based on these characters, one may define the Dirichlet/Fej\'er means for an integrable function $f$ that is {\it constant} in each conjugacy class: 
\begin{align}\label{defFejermeans}S_nf(y):=\sum_{k=0}^{n-1}\hat{f}(k)\chi_k(y) \quad \mathrm{and} \quad \sigma_n f(y):=\frac{1}{n}\sum_{k=0}^{n-1}S_kf(y),\end{align}
where the Fourier coefficients $(\widehat{f}(k))_{k\in \mathbb{N}}$ are defined by $\hat{f}(k):=\int_{G}f(x)\bar{\chi}_k d\mu(x).$

When the group $G$ is abelian, the pointwise convergence problem of the Dirichlet/Fej\'er means has been well studied. For instance, in the  case where $G_k=\mathbb Z_2$ for all $k\geq0$, the resulting character system is the well-known Walsh system. In  this scenario,  the pointwise convergence of Fej\'er means of integrable functions was established by Fine  in  \cite{Fin55}  (see also Schipp \cite{Sch75} for alternative proof), while that of Dirichlet means for square integrable functions was established by Billard \cite{Bil67} after Carleson's resolution of the famous Lusin conjecture for Fourier series on the circle. 
The Walsh system was generalized to $p$-series fields ($G_k=\mathbb Z_p$ where $p$ is a prime), Vilenkin systems ($G_k=\mathbb Z_m$ where $m$ is any fixed integer) and general bounded Vilenkin systems ($G_k=\mathbb Z_{m_k}$). 
Many Mathematicals have explored  Fourier analysis on these system. For example,  Pal and Simon \cite{PaSi77}  extended Taibleson's result \cite{Tai67} on pointwise convergence of Fej\'er means for integrable functions over $p$-series fields to all the bounded Vilenkin systems. Moreover, Gosselim \cite{Gos73} obtained the almost everywhere (abbreviated as {\it a.e.}) convergence of Dirichlet means for functions in $L_p$ with $p>1$ for all bounded Vilenkin systems.
However, the problems on the unbounded Vilenkin systems become slightly more complicated. Firstly, Price's result \cite{Pri57} indicates that one can expect only the {\it a.e.} convergence even for the Fej\'er means of continuous functions. After Young's result \cite{You96} on {\it a.e.} convergence of any lacunary Dirichlet means, G\'at \cite{Gat99} obtained the {\it a.e.} convergence of the Fej\'er means for functions in $L_p$ with $p>1$. Later on, G\'at \cite{Gat03} extended his previous result to the case $p=1$ but only for canonical lacunary Fej\'er means.  The {\it a.e.} convergence of Dirichlet means for $p>1$ and that of Fej\'er means for $p=1$ seem to remain open for unbounded Vilenkin systems. 

In the non-abelian case of $G$, the problem of   pointwise convergence becomes quite subtle with only one positive non-trivial result as far as the authors are aware of. The reason behind this complexity might be explained as follows:  most of the harmonic analysis methods for abelian groups rely on the fact that the absolute value of characters is $1$. However,  in non-abelien groups, the resulting character might be unbounded or could even take a value of zero. This was illustrated in \cite[Page 29]{Gat06} for the special case where $G_k=S_3$ is the permutation group for all $k\in\mathbb N$. 
In the aforementioned paper \cite{Gat06}, G\'at obtained {\it the only positive result} in the literature, namely, the Fej\'er means is {\it a.e.} convergent for square integrable functions that are constant on each conjugacy class. Additionally, in the same paper,  he  established  that for any $1<p<2$, there exist some bounded non-abelian group $G$ and some $f\in L_p(G)$ such that the Fej\'er means of $f$ is not {\it a.e.} convergent.

On the other hand, within the realm of noncommutative harmonic analysis, motivated by operator algebra, quantum probability and noncommutative ergodic theory,  there has been extensive research on maximal inequalities and the almost uniform convergence. For a more detailed exploration of the development of noncommutative maximal inequalities, readers may find valuable insights in \cite{Jun02, JuXu07, HLW21, HWW, HLX23, HRW, CXY13}.

Motivated by all above, we  aim to establish a noncommutative analogue of G\'at's result \cite{Gat06} for any totally disconnected compact bounded group. More precisely, we will establish a noncommutative maximal inequality for Fej\'{e}r means of square integrable functions with values in noncommutative $L_2$-space, and subsequently deduce the almost uniform convergence in a standard manner.

In order to state the main result, we introduce additional notations. Consider a semifinite  von Neumann algebra $\mathcal M$ equipped with a normal semifinite faithful trace $\tau$. The associated noncommutative $L_p$-spaces are denoted by $L_p(\mathcal M)$. The vector-valued noncommutative $L_p$-spaces $L_p(\mathcal{M};\ell_\infty)$ will enable us to formulate $L_p$-norm of the classical maximal function. Our focus in this paper is primarily on the tensor von Neumann algebra $\mathcal N=L_\infty(G, \mu)\bar{\otimes}\mathcal M$ with the canonical tensor trace. Any element in $L_p(\mathcal N)$ can be viewed as an $L_p(\mathcal M)$-valued function that is $p$-integrable on $G$.  Let $S_{\mathcal M}$ be the dense $\ast$-subalgebra of $\mathcal M$ in which the trace of the support of every element is finite. For $S_{\mathcal M}$-valued simple function $f$ that is constant in each conjugacy class,  the Fej\'er means can be defined without any ambiguity as following
$$\sigma_n f(y):=\frac{1}{n}\sum_{k=0}^{n-1}S_kf(y).$$
This definition then extends to the $L_2(\mathcal N)$ spaces thanks to the Hilbert-valued Plancherel theorem. 

Because any character cannot distinguish between group elements within the same conjugacy class, in this paper, {\it we exclusively focus on functions that keep a constant value within each conjugacy class.}

Our main result reads as follows.

\begin{theorem}\label{1.1}
	Let $f \in L_2(\mathcal{N})$. There is a universal constant $C$ such that 
	$$\left\|(\sigma_nf)_n    \right\|_{L_2(\mathcal N; \ell_\infty)}  \leq C\rho(G)^2\Vert    f \Vert_2,$$
	and $\sigma_nf\rightarrow f$ bilaterally almost uniformly as $n\rightarrow\infty$.
\end{theorem}

Recalling $\rho(G)=\sup_{k\geq0}m_k$, it becomes intriguing to explore whether a similar maximal inequality holds for unbounded totally disconnected non-abelian groups, namely when $\rho(G)=\infty$. In a work in progress  \cite{HWW2}, inspired by geometric group theory, especially quantum groups, we are able to provide the Fej\'er type means for all the totally disconnected compact groups that uphold the noncommutative maximal inequalities for all $1<p\leq\infty$.

We will focus on the proof of the maximal inequality in Theorem \ref{1.1}, since the bilateral uniform convergence follows from a standard density argument in the noncommutative setting (refer to, e.g. \cite{CXY13, HWW}) combined with the well-known fact that the character system  constitutes a complete orthonormal basis. For this reason, we will not revisit the definition of bilateral almost uniform convergence here and refer the reader to e.g. \cite{JuXu07} for information.



\section{Preliminaries and the key intermediate estimates}\label{section 2}

In this section, we first recall some preliminaries in noncommutative analysis, and then present the key intermediate estimates crucial for proving of Theorem \ref{1.1}.

 \subsection{Vector-valued non-commutative \boldmath{$L_p$}-spaces }
Let $\cm$ be  a von Neumann algebra equipped with a normal semifinite faithful trace $\tau$, and $S_{\cm}^+$ be the set of all positive $x\in\cm$ such that
$\tau(s(x))<\infty,$ where $s(x)$ denotes the support of $x$, that is, the least projection $e\in\cm$ such that $exe=x.$
Define $S_\cm$ as the linear span of $S_{\cm}^+$. For any $p\in(0,\,\infty)$, we define
$$\|x\|_{p}:=(\tau|x|^p)^{1/p}, \ \ \ x\in S_\cm,$$
where $|x|:=(x^*x)^{1/2}$. The completion
of $(S_\cm, \|\cdot\|_{p})$ is denoted by $L_p(\cm)$ which is the usual  noncommutative $L_p$-space
associated with $(\cm,\,\tau)$. For $p=\infty$, we typically define $L_{\infty}(\cm)=\cm$ equipped with the operator norm $\|\cdot\|_{\cm}$. 
Let $L_p^+(\cm)$ denote the positive part of $L_p(\cm)$. Recalling that $L_2(\mathcal{M})$ is a Hilbert space with respect to inner product $\langle x,\,y\rangle  = \tau(x^*y)$. We refer to \cite{PiXu03} for more information on noncommutative $L_p$-spaces.

Next we recall briefly the definition and fundamental properties of the vector-valued noncommutative $L_p$-space $L_p(\cm ; \ell_\infty)$ and refer the reader to \cite{Pis93, Jun02, JuXu07} for more details. Let $1\leq p\leq\infty$. This space consists of all the sequences $x= (x_k)_{k \ge 1}$ in $L_p(\cm)$ which admit a factorization of the following form: there are $a,b$  and a bounded sequence $y=(y_k)_{k \ge 1} \subset L_{\infty}(\mathcal{M})$ such that 
$$x_k= a y_k b, \quad \forall k \ge1.$$
We then define
$$\Vert (x_k)_k \Vert_{L_p(\mathcal{M};\ell_\infty)}
=\inf \left\lbrace  \sup_{k \ge 1} \Vert y_k \Vert_\infty \Vert a \Vert_{2p}  \Vert b \Vert_{2p} \right\rbrace 
.$$
 where the infimum is taken over all factorizations of $x$ as above. More general, for any index set $\mathcal{I} $, we can define $L_p(\cm;\ell_\infty(\mathcal{I}))$ and 
 $$\Vert (x_\lambda)_{\lambda \in \mathcal{I}} \Vert_{L_p(\mathcal{M};\ell_\infty)}  = \inf_{x_\lambda =ay_\lambda b} \left\lbrace  \sup_{\lambda \in \mathcal{I}}\Vert y_\lambda \Vert_\infty \Vert a \Vert_{2p} \Vert b\Vert_{2p}   \right\rbrace. $$
It has been checked in \cite{JuXu07} that if $x \in L_p(\cm;\ell_\infty(\mathcal{I})) $, then 
$$\Vert (x_\lambda)_{\lambda \in \mathcal{I}} \Vert_{L_p(\mathcal{M};\ell_\infty)}  =\sup_{  \Lambda \subset \mathcal{I}} \left\lbrace   \left\| (x_\lambda)_ {\lambda \in \Lambda}\right\|_p   : \Lambda  ~ \text{is} ~ \text{finite}     \right\rbrace .  $$

\begin{remark}\label{2.1}
	In \cite{JuXu07}, the authors also provided an intuitive description of the maximal norm for $x=(x_k)_k \subset L_p^+(\cm)$. Precisely, $x \in L_p(\cm;\ell_{\infty})$ if and only if there exists some $y \in L_p^+(\cm)$ such that $x_k \le y$ for all $k \ge 1,$
	and moreover
	$$\left\|   (x_k)_k \right\|_{L_p(\mathcal{M};\ell_\infty)} = \inf \{ \Vert   y \Vert_p : x_k \le y ~~  \forall k \ge 1\}.$$
\end{remark}
For this reason, we will often use the notation $\left\| { \sup_{k}}^{+}x_k  \right\|_p$
 to represent $\Vert (x_k)_k \Vert_{L_p(\mathcal{M};\ell_\infty)}$. However, it is important to caution the reader  that $\left\| { \sup_{k}}^{+}x_k  \right\|_p$ is just a notation since $\sup_k x_k$ may not make sense in the noncommutative setting.

\begin{remark}\label{2.3}
From {Remark \ref{2.1}}, it follows easily that $0 \le x_n \le y_n $ implies $$\left\|  {\sup_n}^+ x_n \right\|_p \le \left\|  {\sup_n}^+ y_n \right\|_p,$$
which will be used frequently in this paper.
\end{remark}

	

\subsection{Martingale structure and the key estimates}


In this subsection, we recall some necessary preliminaries  for proving Theorem \ref{1.1} (see e.g. \cite{Gat06,GaTo96,Gat99}). Specifically, we include the construction of a martingale structure on the totally disconnected compact groups which will reduce Theorem \ref{1.1} to an estimate for the ``dyadic block", as detailed in Proposition \ref{1.3}.

Consider	$y=(y_k)_{k \in \mathbb{N}}$ as an element in $G$, where $y_k \in G_k$. We define  $I_0(y)=G $ and for $k \ge 1$,
$$ I_k(y) := \{x \in G:x_j=y_j,0 \le j \le k-1 \}.$$  
Let $ \mathcal{F}_k$ be the $\sigma$-algebra generated by $\{I_k(x): x \in G\}$, and $E_k$ be the associated conditional expectation. 
It is a well-known fact that $E_k$'s extend to conditional expectations associated to the von Neumann subalgebras $L_\infty(\mathcal F_k)\bar{\otimes}\mathcal M$'s and we continue to  denote them by $E_k$ without causing confusion. The resulting martingale difference is denoted by $d_k=E_k-E_{k-1}$ with convention $E_{-1}=0$. Then from  Doob's maximal inequalities for noncommutative martingales \cite{Jun02}, one has
$$\left\|  {\sup_{k }}^+ E_kf    \right\|_p \le C_p\| f \|_p,\quad\forall f\in L_p(\mathcal N),$$
where $C_p$ is a constant only depending on $p$. 
In particular, for $p=2$, after checking the related arguments in \cite[Lemma 3.1, Lemma 3.9 and Remark 5.5]{Jun02}, one gets
\begin{align}\label{2.40}\left\|  {\sup_{k }}^+ E_kf    \right\|_2 \le 2\| f \|_2,\quad\forall f\in L_2(\mathcal N).\end{align}
Moreover, due to the orthogonality of martingale differences, we also have
\begin{align}\label{2.4}\left\|\Big(\sum_k|d_kf|^2\Big)^{\frac12} \right\|_2=\left\|\Big(\sum_k|(d_kf)^*|^2\Big)^{\frac12} \right\|_2=\|f\|_2.\end{align}
For $1<p\neq2<\infty$, a noncommutative version of Burkholder-Gundy inequality is presented in \cite{PiXu97}, but it will not be utilized in this paper. For this reason, {\it in the sequel we will only state all necessary estimates  for the case $p=2$.}

By exploiting the martingale estimates  \eqref{2.40} and \eqref{2.4}, Theorem \ref{1.1} will be reduced to the following estimate for each ``dyadic block". For the details of the reduction, we refer to Section \ref{section 3}. Recalling that $n=\sum^\infty_{k=0}n_kP_k$, we define
$$|n|:=\max\{ k:n_k \not=0 \},$$
representing the largest index $k$ for which $n_k$ is non-zero.

\begin{Proposition}\label{1.3}
	Let $A\in\mathbb N$ and $f \in L_2(\mathcal{N}) $, then we have
	\begin{align}\label{es1.3}\left\| {\sup_{ |n| =A }}^+ \sigma_nf \right\|_2 \le 640\rho(G)^2\Vert f \Vert_2.\end{align}
\end{Proposition}	

Recall that $P_n$ is defined by $P_0=1, P_{k+1}=p_kP_k$. It follows that $P_{|n|} \le n <P_{|n| +1}$, thus
$$\{n\in\mathbb N:|n|=A\}=\{n\in\mathbb N:P_{A} \le n <P_{A+1}\}.$$
When $p_k=2$ for all $k\geq0$, the set $\{n\in\mathbb N:|n|=A\}$ corresponds precisely to the dyadic block $\{n\in\mathbb N: 2^A \le n <2^{A+1}\}$.

The estimation process for the `dyadic block', as outlined in Proposition \ref{1.3}, necessitates the introduction of additional notation and relies on a crucial intermediate estimate. Specifically, we define  the Dirichlet kernels and Fej\'er kernels as follows:
\begin{align}\label{defFejerkernel}D_n(y,x):= \sum_{k=0}^{n-1}\chi_k(y) \bar{\chi}_k(x) \quad \mathrm{and} \quad  K_n(y,x):=\frac{1}{n}\sum_{k=0}^{n-1}D_k(y,x)\end{align}
for $n\geq1$ and set $D_0:=K_0:=0$ for convenience.
Given two positive integers $m$ and $n$, we then define
$$K_{m,n}(y,x) \coloneqq  \sum_{k=m}^{m+n-1}D_k(y,x).$$
Furthermore, we introduce an additional notation
$$ n^{(s)}= \sum_{k=s}^{\infty} n_kP_k.$$ 
Then the key estimate can be stated as follows.		
	
\begin{Proposition}\label{1.2}
	Let $f \in L_2(\mathcal{N}) $ and $ A \ge s \in \mathbb{N}$, then we have
	\begin{align}\label{es1.2}\left\|  {\sup_{\vert n \vert =A}}^+ \int_{G} f(x)K_{n^{(s)},P_s}(y,x) d\mu(x)  \right\|_2   \le 40\rho(G)^2P_s((A-s)^2+1) \Vert f \Vert_2.\end{align}
\end{Proposition}

By invoking a nice formula referenced in \eqref{formula}, transitioning from Proposition \ref{1.2} to Proposition \ref{1.3} is straightforward.  Detailed discussions of this transition can be found in Section \ref{sec4}.  However, the proof of Proposition \ref{1.2} is extremely involved, and takes up a substantial part of the paper. 

Before  presenting the delicate arguments for Proposition \ref{1.2}, which will be elaborated in Section \ref{section 5}, we first introduce the following relatively easier intermediate estimate essential for the proof. Consider an element $y=(y_\ell)_{\ell\in\mathbb N}$ in $G$.
For each $\ell \in \mathbb{N}$, the conjugacy class of an element $y_\ell \in G_\ell$ is denoted by $j_\ell(y_\ell)$, given by
$$j_\ell(y_\ell):=\{ x_\ell \in G_\ell :\exists z_\ell \in G_\ell ,y_\ell = -z_\ell+x_\ell+z_\ell \}.$$  
Furthermore, we set $J_0(y)=G$ and for $n\geq1$, define
$$ J_n(y) :=\{ x \in G : x_i \in j_i(y_i),0 \le i \le n-1 \}= j_0(y_0)\times \cdots \times j_{n-1}(y_{n-1})\times \bigtimes_{i=n}^{\infty }G_i.$$

\begin{lemma} \label{4.4}
	Let $f \in L_2(\mathcal{N})$ and $ A \ge s \in \mathbb{N}$, then we have that 
	\begin{align}\label{es4.4} \left\|  {\sup_{|n| =A}}^+     \frac1{\mu(J_s(\cdot))}\int_{J_s(\cdot)}f(x)\chi_{n^{(s)}}(\cdot) \bar{\chi}_{n^{(s)}}(x) d\mu(x)     \right\|_2 \le \|f\|_2. \end{align}
\end{lemma}


To prove the reductions or the estimates such as \eqref{es1.3}, \eqref{es1.2}, \eqref{es4.4}, we primarily adhere to the classical sketch outlined in \cite{Gat06}. Nevertheless, due to the noncommutativity, we have to pay meticulous attention to the pointwise estimates, and it compels us to revise many classical arguments. Notably, in reverting to the classical case,  our results  exhibit enhanced precision, particularly in the bounds of certain estimates, which are  more accurate, exemplified by being of order $\rho(G)^2$. This refinement in accuracy, especially in noncommutative settings, underscores the nuanced improvements our approach offers over traditional methodologies.




As described  above, the paper is organized as follows. In Section~\ref{section 3}, we prove Theorem~\ref{1.1} by assuming Proposition~\ref{1.3}. In Section~\ref{sec4}, we derive Proposition~\ref{1.3} based on Proposition~\ref{1.2}.
Section~\ref{section 4} presents  the  relatively easier intermediate estimate~\eqref{4.4}.  Finally,  in Section~\ref{section 5}, we provide the proof of Proposition~\ref{1.2}.

In this paper,  the letter $C$ denotes a positive absolute constant that  not inevitably the same one in each occurrence. The notation  $Y \lesssim Z$ or $Z \lesssim Y$ indicates that there exists an absolute constant $C>0$ such that $Y \le CZ$ or $Z \le CY$, respectively.

\section{Proof of Theorem \ref{1.1}: based on Proposition \ref{1.3}}\label{section 3}
Assuming Proposition \ref{1.3}, we conclude Theorem \ref{1.1} in this section. While we fundamentally adhere to the classical argument presented in \cite{Gat06}, adapting it to the noncommutative setting requires careful consideration. In particular, our final argument is slightly different from the classical one due to the noncommutativity. To support these adjustments, several lemmas prove to be indispensable.

The following lemma admits an $L_p$-version for all $1\leq p\leq\infty$. However, for the purposes of this paper, we focus exclusively on the case $p=2$.  For additional results and their proofs, we refer to \cite[Proposition 2.3]{HWW}.

\begin{lemma}\label{3.2}
	Let $(x_n) \in \ell_2(L_2(\mathcal M))$, then
	$$\left\|  {\sup_{n}}^+ x_n    \right\|_2  \le 
	\left( \sum_{n} \Vert x_n \Vert_2^2 \right)^{\frac{1}{2}}.$$
\end{lemma}

The following lemma is well-known to experts. But we include a proof here to ensure completeness.

\begin{lemma}\label{3.3} 
	Let $\{ a_n\} \subseteq \mathbb{C}$ and $(f_n)_n \in L_2(L_\infty(G)\overline{\otimes} \cm; \ell_\infty)$, then one has
	$$\left\|  {\sup_{n}}^+ a_nf_n    \right\|_2 \le \sup_{n}|a_n| \left\|  {\sup_{n}}^+ f_n    \right\|_2 . $$
\end{lemma}

\begin{proof}
	According to the definition of $L_2(\cm; \ell_\infty)$, for any $\varepsilon >0$, there is a decomposition $f_n =ay_nb$, for any $n \ge 1$ such that
	$$\Vert a \Vert_4 \sup_{ n }\Vert y_n \Vert_{\infty} \Vert b \Vert_4 \le \left\|{\sup_{n}}^+ f_n    \right\|_2 +\varepsilon . $$
	Hence $a_nf_n= aa_ny_n b$  and we can further express it as:
	$$\Vert a \Vert_4 \sup_{ n } \Vert a_ny_n \Vert_{\infty} \Vert b \Vert_4 \le \Vert a \Vert_4 \sup_{ n }\vert a_n \vert \sup_{ n } \Vert y_n \Vert_{\infty} \Vert b \Vert_4
	\le \sup_{ n }\vert a_n \vert \left( \left\|{\sup_{n}}^+ f_n    \right\|_2 +\varepsilon \right). $$
	By letting $\varepsilon \to 0 $,  we  have
	$$\left\|  {\sup_{n}}^+ a_nf_n    \right\|_2 \le \sup_{n}\vert a_n \vert \left\|  {\sup_{n}}^+ f_n    \right\|_2  . $$
\end{proof}

The following two lemmas have been established in \cite{Gat06} in the scalar-valued case, we present them here with a proof to facilitate the reader's understanding on the proof of Theorem \ref{1.1}.

\begin{lemma}\label{3.1}
	Let $f \in L_2(\mathcal{N})$ and $ n \in \mathbb{N}$ such that $ \vert n \vert =A$. Then $ \sigma_n (f) = \sigma_n(E_{A+1} f).$
\end{lemma}

\begin{proof}
By Lemma \ref{4.3}, 
 we have
\begin{align}\label{EnChi}E_{A+1}f(x)=S_{P_{A+1}}f(x)=\sum_{j=0}^{P_{A+1}-1} \hat{f}(j) 
		 \chi_j(x).\end{align}
Then by the complete orthogonality property of the character system and $P_A \le n < P_{A+1}$, we have that
	for $i< P_{A+1}$, 
	\begin{flalign*}
		\int_{G} E_{A+1}f(x)\bar{\chi}_i(x) d\mu(x) &=\int_{G} \sum_{j=0}^{P_{A+1}-1} \hat{f}(j) 
		 \chi_j(x)\bar{\chi}_i(x) d\mu(x)\\ 
		 &= \int_{G}  \hat{f}(i)  \chi_i(x)  \bar{\chi}_i(x) d\mu(x) \\
		 &=\hat{f}(i) \Vert \chi_i \Vert_{L_2(G)}^2 \\
		 &=\int_{G} f(x)\bar{\chi}_i(x) d\mu(x).
	\end{flalign*}
Recalling the definition of Fej\'er means \eqref{defFejermeans}, then we immediately obtain the desired identity.
\end{proof}

\begin{lemma}\label{3.4}
	 Let $f \in L_2(\mathcal{N})$ and $n\in\mathbb N$ such that $|n|= A $, then 
	\begin{flalign}\label{e3.1}
		n\sigma_{n}(E_Af)=P_A \sigma_{P_A}(E_Af)+(n-P_A)E_Af.
	\end{flalign}
	Moreover,
	$$P_A\sigma_{P_A}(E_Af)=P_A\sigma_{P_A}(E_Af-E_{A-1}f)+P_{A-1}\sigma_{P_{A-1}}(E_{A-1}f)+(P_A-P_{A-1})E_{A-1}f .$$
	
\end{lemma}
\begin{proof}
	From Identity~\eqref{EnChi}, it is obvious that
	$$\int_{G}E_Af(x) \bar{\chi}_j(x) d\mu(x) = 0 , \qquad \forall j \ge P_A .$$
	So for $P_A\leq k\leq n-1$, $S_k(E_Af)=S_{P_A}E_Af=E_Af$, and thus
	\begin{align*}
	n\sigma_n(E_Af) &=\sum^{P_A-1}_{k=0}S_k(E_Af)+\sum^{n-1}_{k=P_A}S_k(E_Af)\\
	&=P_A\sigma_{P_A}(E_Af)+(n-P_A)E_Af.
	\end{align*}
	This is Identity~\eqref{e3.1}. 
		
	To obtain the second identity, first observe that
	$$P_A\sigma_{P_A}(E_Af)=P_A\sigma_{P_A}(E_Af-E_{A-1}f) +P_A\sigma_{P_A}(E_{A-1}f).$$
	Then the application of \eqref{e3.1} to $n = P_A$ and $E_{A-1}f$ yields the desired identity. 
	\end{proof}

We are now in the position to prove Theorem \ref{1.1}.

\begin{proof}[\bf{Proof of Theorem \ref{1.1}}]
From Lemma \ref{3.1}, we have
$$\left\|  {\sup_{n\geq1}}^+ \sigma_nf    \right\|_2=  \left\|  {\sup_{A\geq0}}^+ {\sup_{\vert n \vert  =A}}^+ \sigma_nf   \right\|_2  =\left\|  {\sup_{A\geq0}}^+ {\sup_{\vert n \vert  =A}}^+ \sigma_n(E_{A+1}f)   \right\|_2.$$
Based on Lemma \ref{3.2} and Lemma \ref{3.4}, it is established that
	\begin{flalign*}
		\left\|  {\sup_{A}}^+ {\sup_{\vert n \vert  =A}}^+ \sigma_n(E_{A+1}f)   \right\|_2  
		&\le \left\|  {\sup_{A}}^+ {\sup_{\vert n \vert  =A}}^+ \sigma_n(E_{A+1}f-E_Af )   \right\|_2 +\left\|  {\sup_{A}}^+ {\sup_{\vert n \vert  =A}}^+ \sigma_n(E_Af)   \right\|_2  \\
		&\le  \left(  \sum_{A} \left\| {\sup_{\vert n \vert  =A}}^+ \sigma_n(E_{A+1}f-E_Af)   \right\|_2^2 \right)^{\frac{1}{2}}  \\
		&\ \ \ \quad+
		\left\|  {\sup_{A}}^+ {\sup_{\vert n \vert  =A}}^+ \frac{P_A}{n}\sigma_{P_A}(E_A f)+\frac{n-P_A}{n}E_Af   \right\|_2  \\
		&\eqqcolon \uppercase\expandafter{\romannumeral1}+\uppercase\expandafter{\romannumeral2}.
	\end{flalign*}
We proceed by estimating the term $\uppercase\expandafter{\romannumeral1}$ first.  According to Proposition \ref{1.3},
$$\uppercase\expandafter{\romannumeral1}\le 640\rho(G)^2\left(  \sum_{A}  \Vert E_{A+1}f-E_Af \Vert_2^2 \right)^{\frac{1}{2}}.$$
Following the  to the martingale result \eqref{2.4}, we find
\begin{flalign*}
	\left( \sum_{A}  \Vert E_{A+1}f-E_Af \Vert_2^2\right)^{\frac{1}{2}}&=  \left\| \left(  \sum_{A} \vert E_{A+1}f-E_Af \vert^2 \right)^{\frac{1}{2}}  \right\|_2 \\
	&\le \Vert f \Vert_2.
\end{flalign*}
Consequently,
$$\uppercase\expandafter{\romannumeral1}\le 640 \rho(G)^2 \Vert f \Vert_2.$$

Next, we turn our attention to estimate the term $\uppercase\expandafter{\romannumeral 2}$. By the triangle inequality, we get
$$\uppercase\expandafter{\romannumeral2} \le\left\|  {\sup_{A}}^+ {\sup_{\vert n \vert  =A}}^+ \frac{n-P_A}{n}(E_Af)   \right\|_2 + 	\left\|  {\sup_{A}}^+ {\sup_{\vert n \vert  =A}}^+ \frac{P_A}{n} \sigma_{P_A}(E_Af)   \right\|_2,$$
	and utilizing noncommutative Doob's inequality \eqref{2.40} and along with Lemma~\ref{3.3}, we find
$$\uppercase\expandafter{\romannumeral2} \le  2\Vert    f \Vert_2 + 	\left\|  {\sup_{A}}^+ {\sup_{\vert n \vert  =A}}^+  \sigma_{P_A}(E_Af)   \right\|_2= 2\Vert    f \Vert_2 + 	\left\|  {\sup_{A}}^+  \sigma_{P_A}(E_Af)   \right\|_2.$$
Now it suffices to estimate the second term. By applying Lemma \ref{3.3}, Lemma $\ref{3.4}$, the triangle inequality and the fact that $p_k \ge 2$ for each  $k \in  \mathbb{N}$, we deduce
 \begin{flalign*}
 \left\|  {\sup_{A}}^+  \sigma_{P_A}(E_Af)   \right\|_2 
 	&\le \left\|  {\sup_{A}}^+  \frac{P_{A-1}}{P_A}\sigma_{P_{A-1}}(E_{A-1}f)   \right\|_2 +
 	 \left\|  {\sup_{A}}^+  \frac{P_A-P_{A-1}}{P_A}E_{A-1}f   \right\|_2\\
	 &\quad\quad +
 	  \left\|  {\sup_{A}}^+  \sigma_{P_A}(E_Af-E_{A-1}f)   \right\|_2  \\
 	&\le \frac{1}{2}\left\|  {\sup_{A}}^+  \sigma_{P_{A-1}}(E_{A-1}f)   \right\|_2 +
 	\left\|  {\sup_{A}}^+  E_{A-1}f   \right\|_2    \\
 	&\quad\quad  + \left\|  {\sup_{A}}^+  \sigma_{P_A}(E_Af-E_{A-1}f)   \right\|_2.  
 \end{flalign*}

Since $ \vert P_A \vert =A, P_A \in \{n: \vert n \vert =A \}$,  and according to the result obtained for  the term  $\uppercase\expandafter{\romannumeral1}$, we establish
$$\left\|  {\sup_{A}}^+  \sigma_{P_A}(E_Af-E_{A-1}f)   \right\|_2 \le \left\|  {\sup_{A}}^+ {\sup_{\vert n \vert  =A}}^+ \sigma_{n}(E_Af-E_{A-1}f)   \right\|_2 \le 640 \rho(G)^2 \Vert f \Vert_2.$$
This leads us to the following estimate:
 \begin{flalign*}
 	\left\|  {\sup_{A}}^+  \sigma_{P_A}(E_Af)   \right\|_2 &\le \frac{1}{2}\left\|  {\sup_{A}}^+  \sigma_{P_{A-1}}(E_{A-1}f)   \right\|_2 +
 	(2+640\rho(G)^2)\Vert f \Vert_2  \\
 	&\le \frac{1}{2}\left\|  {\sup_{A}}^+  \sigma_{P_{A}}(E_{A}f)   \right\|_2 +
 	(2+640\rho(G)^2)\Vert f \Vert_2. 
 \end{flalign*}
Therefore, we obtain
 $$ \left\|  {\sup_{A}}^+  \sigma_{P_{A}}(E_{A}f)   \right\|_2  \le
 2(2+640\rho(G)^2)\Vert f \Vert_2.$$
Consequently, 
$$\left\|  {\sup_{n \in \mathbb{Z}_+}}^+ \sigma_nf    \right\|_2 \le \uppercase\expandafter{\romannumeral1}+\uppercase\expandafter{\romannumeral2} \lesssim \rho(G)^2\Vert f \Vert_2.$$
\end{proof}

\section{Proof of Proposition \ref{1.3}: based on Proposition \ref{1.2}}\label{sec4}
Assuming the validity of Proposition \ref{1.2}, we consequently deduce  Proposition \ref{1.3} in this section.

\begin{proof}[\bf{Proof of Proposition \ref{1.3}}]
	
	Observe that for $n$ when $\vert n \vert =A$, the following equality holds  
	\begin{align}\label{formula}nK_n(y,x) = \sum_{s=0}^{A} \sum_{j=0}^{n_s-1} K_{n^{(s+1)}+jP_s,P_s}(y,x). \end{align}
	This leads to
	$$\left\|  {\sup_{\vert n \vert =A}}^+  \sigma_nf \right\|_2 
	\le \frac{1}{n}\sum_{s=0}^{A} \sum_{j=0}^{n_s-1} \left\|  {\sup_{\vert n \vert =A}}^+ \int_{G} f(x) K_{n^{(s+1)}+jP_s,P_s}(y,x) d\mu(x) \right\|_2 . $$
	Considering each fixed $j$ $(j \in \{ 0,\cdots , n_s-1\})$, we obtain
	$$\left\|  {\sup_{\vert n \vert =A}}^+ \int_{G} f(x) K_{n^{(s+1)}+jP_s,P_s}(y,x) d\mu(x) \right\|_2
	\le \left\|  {\sup_{\vert n \vert =A}}^+ \int_{G} f(x) K_{n^{(s)},P_s}(y,x) d\mu(x) \right\|_2 . $$
	Combining Proposition \ref{1.2} and the fact that $n_s \le p_s $ for any $s \in \mathbb{N}$, we get
	\begin{flalign*}
		\left\|  {\sup_{\vert n \vert =A}}^+  \sigma_nf \right\|_2  &\le 40\rho(G)^2\frac{n_s}{P_A}\sum_{s=0}^{A}  P_s  ((A-s)^2+1) \Vert f \Vert_2  \\
		 &\le 40\rho(G)^2 \sum_{s=0}^{A}  \frac{p_sP_s}{P_A}  ((A-s)^2+1) \Vert f \Vert_2  \\
		&= 40\rho(G)^2\sum_{s=0}^{A}  \frac{1}{p_{s+1}\cdots p_{A-1}}  ((A-s)^2+1) \Vert f \Vert_2 \\
		&\le 80 \rho(G)^2 \sum_{s=0}^{A}  \frac{1}{2^{A-s}}  ((A-s)^2+1) \Vert f \Vert_2 \\
		&= 80\rho(G)^2\sum_{i=0}^{A}  \frac{i^2+1}{2^{i}}   \Vert f \Vert_2 \\
		&\le 80\rho(G)^2 \sum_{i=0}^{\infty}  \frac{i^2+1}{2^{i}}   \Vert f \Vert_2 \leq640\rho(G)^2 \Vert f \Vert_2  .
	\end{flalign*}
\end{proof}

\section{Proof of Lemma \ref{4.4}}\label{section 4}
The purpose of this section is to prove Lemma \ref{4.4}, which will play an important role in the next section to demonstrate Proposition \ref{1.3}. To achieve this,  we introduce three pivotal lemmas to manipulate the Dirichlet kernels. These lemmas can be found in the literature. For instance, in \cite{Gat06} the author listed them and  cited the original  references. To enhance comprehension, we present their proofs herein.

\begin{lemma} \label{4.1}
	For all $n \in \mathbb{N}$, there holds
	$$D_n(y,x)=\sum_{\ell =0}^{+\infty}D_{P_\ell}(y,x)\sum_{i=0}^{n_\ell -1}r_\ell^i(y) \bar{r}_\ell ^i(x) \chi_{n^{(\ell+1)}}(y)\bar{\chi}_{n^{(\ell+1)}}(x).$$
\end{lemma}

\begin{proof}
	Without loss of generality, one may assume that $|n|=A$, that is $n=n_0P_0+\cdots + n_AP_A$. Rewriting the Dirichlet kernel, we reformulate it as follows:
	$$D_n(y,x)= \sum_{k=0}^{n-1}\chi_k(y)\bar{\chi}_k(x)=\left( \sum_{k=0}^{P_A-1}+\sum_{k=P_A}^{2P_A-1}+\cdots+\sum_{k=(n_A-1)P_A}^{n_AP_A-1}
	+\sum_{k=n_AP_A}^{n-1}	\right)   \chi_k(y) \bar{\chi}_k(x) . $$
	From the second term to the penultimate one, the summing indices $k$ can be rewritten as $k=k'+1\cdot P_A$, $\dotsm$,  $k=k'+(n_A-1)\cdot P_A$, where  $k'=k_0P_0+\cdots+k_{A-1}P_{A-1}$ ranging from $0$ to $P_A-1$. Thus
	$$\sum_{k=P_A}^{2P_A-1}\chi_k(y)\bar{\chi}_k(x)=\sum_{k=0}^{P_A-1}	\chi_k(y)\bar{\chi}_k(x) r_A^1(y) \bar{r}_A^1(x),$$
	$$\dotsm,$$
	$$\sum_{k=(n_A-1)P_A}^{n_AP_A-1} \chi_k(y) \bar{\chi}_k(x)=
	\sum_{k=0}^{P_A-1}\chi_k(y)\bar{\chi}_k(x)r_A^{n_A-1}(y) \bar{r}_A^{n_A-1}(x).
	$$
	Regarding the last term, by representing the summation index index $k$ as $k=k'+n_AP_A$ with $k'$ ranging from $0$ to $n-1-n_AP_A$,  it is deduced that
	$$ \sum_{k=n_AP_A}^{n-1} \chi_{k}(y) \bar{\chi}_{k}(x)=\left( \sum_{k=0}^{n-1-n_AP_A} \chi_{k}(y) \bar{\chi}_{k}(x)\right)  r_A^{n_A}(y) \bar{r}_A^{n_A}(x)$$
	Therefore, the following is obtained
	$$ D_n(y,x)=D_{P_A}(y,x)\sum_{i =0}^{n_A-1}r_A^i (y)\bar{r}_A^i (x)+
	\left( \sum_{k=0}^{n-1-n_AP_A} \chi_{k}(y) \bar{\chi}_{k}(x)\right)  r_A^{n_A}(y) \bar{r}_A^{n_A}(x).
	$$
	Note that the above summation inside the bracket can be addressed similarly. Repeating the procedure, we conclude finally that
	$$D_n(y,x)=\cdots=\sum_{\ell =0}^{A}D_{P_\ell}(y,x)\sum_{i=0}^{n_\ell -1}r_\ell^i(y) \bar{r}_\ell ^i(x) \chi_{n^{(\ell+1)}}(y)\bar{\chi}_{n^{(\ell+1)}}(x). $$
\end{proof}

\begin{lemma}\label{4.2}
For all $n \in \mathbb{N}$,	
	$$D_{P_n}(y,x)=\mu^{-1}(J_n(y)) \boldsymbol{1}_{J_n(y)}(x)  , \quad  	x,y \in G.	$$	
\end{lemma}

\begin{proof}
This assertion follows from the corresponding identity for a finite group. Let $H$ be a finite group equipped with the canonical Haar measure. Denote by $\Sigma$ the equivalence classes of continuous irreducible unitary representations of $H$, which is still finite. We then claim
	$$\sum_{\sigma \in \Sigma}\chi _\sigma(y) \bar{\chi}_\sigma(x) =\mu^{-1}(j(y))\boldsymbol{1}_{j(y)}(x), $$
	where $\chi _\sigma$ is the character associated to $\sigma\in\Sigma$. Indeed, by the Plancherel theorem, for any complex-valued function $f$, one has 
	$$f(x)=\sum_{\sigma \in \Sigma}\hat{f}(\sigma)\chi_\sigma(x)=\int_{H}f(t)
	\sum_{\sigma \in \Sigma} \chi_\sigma(x) \bar{\chi}_\sigma(t)d\mu(t).
	$$
	Taking $f(x)=\boldsymbol{1}_{j(y)}(x)$, we obtains
\begin{align*}
	\boldsymbol{1}_{j(y)}(x)&=\int_{j(y)} \sum_{\sigma \in \Sigma} \chi_\sigma(x) \bar{\chi}_\sigma(t) d\mu(t)\\
	&=\sum_{\sigma \in \Sigma} \chi_\sigma(x) \int_{j(y)} \bar{\chi}_\sigma(t) d\mu(t)=\mu(j(y))\sum_{\sigma \in \Sigma}
	\chi_\sigma(x) \bar{\chi}_\sigma(y).
	\end{align*}
	
	Now, by applying the above claim repeatedly, we have
	\begin{align*}
		D_{P_n}(y,x)&= \sum_{k=0}^{P_n-1} \chi_k(y) \bar{\chi}_k(x) \\
		&=\sum_{k_0=0}^{p_0-1} \cdots \sum_{k_{n-1}=0}^{p_{n-1}-1}
		\chi_{k_0P_0+\cdots+k_{n-1}P_{n-1}}(y) \bar{\chi}_{k_0P_0+\cdots+k_{n-1}P_{n-1}}(x) \\
		&=\sum_{k_0=0}^{p_0-1} \cdots \sum_{k_{n-1}=0}^{p_{n-1}-1}
		\prod_{\ell =0}^{n-1} r_\ell^{k_\ell}(y) \bar{r}_\ell^{k_\ell}(x)
		= \prod_{\ell =0}^{n-1}\sum_{k_\ell=0}^{p_\ell -1} r_\ell^{k_\ell}(y) \bar{r}_\ell^{k_\ell}(x)   \\
		&=\prod_{\ell =0}^{n-1}\mu_\ell^{-1}(j_\ell(y))\boldsymbol{1}_{j_\ell(y)}(x)
		=\mu^{-1}(J_n(y)) \boldsymbol{1}_{J_n(y)}(x) .
	\end{align*}
\end{proof}

Recall that $m_k$ denotes the order of $G_k$. We define the sequence $(M_k)$ by 
$M_0 :=1$ and $M_{k+1}=m_kM_k$  for $k\geq 1$.
\begin{lemma} \label{4.3}
	For any  $n \in \mathbb{N}$ and $f \in L_2(\mathcal{N})$, the following equality holds:
	$$S_{P_n}f(y)=E_nf(y) . $$
\end{lemma}

\begin{proof}
	Firstly, invoking Lemma~\ref{4.2} yields
	$$S_{P_n}f(y)=\int_{G}f(x)D_{P_n}(y,x) d\mu(x)=\mu^{-1}(J_n(y))\int_{J_n(y)}f(x) d\mu(x).$$
	Recall that $I_0(y)=G $ and for $k \ge 1$,
	$$ I_k(z) := \{x \in G:x_j=z_j,0 \le j \le k-1 \},\quad z \in G.$$  
	This admits $J_ n(y)$ to be decomposed as
	$$J_n(y)=\bigcup_{z: z_i \in j_i(y)  \atop i \in \{ 0,\dotsm,n-1\} }I_n(z).$$
	Since $f$ is constant on each conjugacy class, for any $z$ satisfying $z_i \in j_i(y), i<n$,
	$$\int_{I_n(z)}f(x)d\mu(x) = \int_{I_n(y)}f(x) d\mu(x).$$
	Therefore, the equation rewrites to 
	\begin{align*}
	S_{P_n}f(y)&=\frac{1}{\mu(J_n(y))} \sum_{z:z_i \in j_i(y)  \atop i \in \{ 0,\dots ,n-1\} } \int_{I_n(y)}f(x) d\mu(x)\\
	&=M_n\int_{I_n(y)}f(x)d\mu(x) =\frac{1}{\mu(I_n(y))}\int_{I_n(y)}f(x)d\mu(x) =E_nf(y).
	\end{align*}
\end{proof}


Now, we are ready to show Lemma \ref{4.4}.

\begin{proof}[\bf{Proof of Lemma \ref{4.4}}]
	From the proof of Lemma \ref{4.3}, 
	we have  deduced the following equation
	$$\mu^{-1}(J_s(y)) \int_{J_s(y)} f(x) \bar{\chi}_{n^{(s)}}(x) d\mu(x) = M_s \int_{I_s(y)}f(x)\bar{\chi}_{n^{(s)}}(x) d\mu(x) . $$
	For any element $x \in G$, let us denote it as $x=(x_0,x_1,\cdots ,x_s,\cdots) = (x_{<s},x_{\ge s})$ where $x_{<s}= (x_0, \cdots, x_{s-1})$,  $x_{\ge s}= (x_s,x_{s+1},\cdots)$.
	Applying the above identity, we obtain
	\begin{flalign*}
		&\left\|  {\sup_{\vert n \vert =A}}^+     \mu^{-1}(J_s(y))\int_{J_s(y)} f(x)\chi_{n^{(s)}}(y) \bar{\chi}_{n^{(s)}}(x) d\mu(x)     \right\|_2^2   \\
		=&\left\|  {\sup_{\vert n \vert =A}}^+     M_s\int_{I_s(y)}f(x)\chi_{n^{(s)}}(y) \bar{\chi}_{n^{(s)}}(x) d\mu(x)     \right\|_2^2 \\
		(\text{Definition of $I_s(y)$})=&\left\|  {\sup_{\vert n \vert =A}}^+     M_s\int_{I_s(y)}f(y_{<s},x_{\ge s})\chi_{n^{(s)}}(y) \bar{\chi}_{n^{(s)}}(y_{<s},x_{\ge s}) d\mu(x)    \right\|_2^2 \\
		(\text{Definition of $n^{(s)}$})=&\left\|  {\sup_{n \in  \mathbb{N}_{s,A}}}^+     M_s\int_{I_s(y)}f(y_{<s},x_{\ge s})\chi_{n^{(s)}}(y) \bar{\chi}_{n^{(s)}}(y_{<s},x_{\ge s}) d\mu(x)   \right\|_2^2 \\
		( \text{By Lemma \ref{3.2}})\le&  \sum_{n \in \mathbb{N}_{s,A}}  \left\| M_s \int_{I_s(y)}f(y_{<s},x_{\ge s})\chi_{n^{(s)}}(y) \bar{\chi}_{n^{(s)}}(y_{<s},x_{\ge s}) d\mu(x)  \right\|_2^2 ,
	\end{flalign*}
	where 
	$$\mathbb{N}_{s,A}:= \left\lbrace n \mid { n_i \in \{0,\cdots ,p_i-1\},\  i \in \{s,\cdots ,A\}} \right\rbrace .$$
	For convenience, we denote
	\begin{align}\label{Gt}
		G_{<t>}=&\bigtimes_{i=0}^{t-1 }G_i \times \bigtimes_{i=t+1}^{\infty }G_i, & G_{\ge A+1} =& \bigtimes\limits_{i=A+1}^{\infty }G_i,  \\
		G_{< s} =& G_0 \times \cdots \times G_{s-1}, & G_{s,A} =& G_s \times \cdots \times G_A . \nonumber
	\end{align}
	and set
	\begin{align*}
		d\mu(x_{<s}) = & d\mu_0(x_0) \times \cdots \times d\mu_{s-1}(x_{s-1}), \\
		d\mu(x_{s,A}) = & d\mu_s(x_s) \times \cdots \times d\mu_A(x_A), \\
		d\mu(x_{\ge A+1}) = & d\mu_{A+1}(x_{A+1})  \times \cdots\cdots.
	\end{align*}
	Since $n \in \mathbb{N}_{s,A}$ and the definition of $E_{A+1}f$, we have
	\begin{flalign*}
		&\sum_{n \in \mathbb{N}_{s,A}}  \left\|  \int_{G_{\ge s}}f(y_{<s},x_{\ge s} )\chi_{n^{(s)}}(y) \bar{\chi}_{n^{(s)}}(y_{<s},x_{\ge s}) d\mu(x_{\ge s})  \right\|_{L_2(\cm)}^2 \\
		=&\sum_{n \in \mathbb{N}_{s,A}}  \left\|  \int_{G_{s,A}}(E_{A+1}f)(y_{<s},x_s \cdots x_A,y_{\ge A+1} )\chi_{n^{(s)}}(y) \bar{\chi}_{n^{(s)}}(y_{<s},x_{\ge s}) d\mu(x_{s,A})  \right\|_{L_2(\cm)}^2.
	\end{flalign*}
	Noting $\mu(I_s(y))=1/M_s$, 
	we can thus get
		\begin{align*}
			&\sum_{n \in \mathbb{N}_{s,A}}  \left\| M_s \int_{I_s(y)}f(y_{<s},x_{\ge s} )\chi_{n^{(s)}}(y) \bar{\chi}_{n^{(s)}}(y_{<s},x_{\ge s}) d\mu(x)  \right\|_{L_2(G; L_2(\cm))}^2 \\
			=& \int_{G} \sum_{n \in \mathbb{N}_{s,A}}  \left\| M_s \int_{I_s(y)}f(y_{<s},x_{\ge s} )\chi_{n^{(s)}}(y) \bar{\chi}_{n^{(s)}}(y_{<s},x_{\ge s}) d\mu(x)  \right\|_{L_2({\cm})}^2 d\mu(y) \\
			=&\int_{G} \sum_{n \in \mathbb{N}_{s,A}}  \left\|  \int_{G_{\ge s}}f(y_{<s},x_{\ge s} )\chi_{n^{(s)}}(y) \bar{\chi}_{n^{(s)}}(y_{<s},x_{\ge s}) d\mu(x_{\ge s})  \right\|_{L_2(\cm)}^2 d\mu(y) \\
			=&\int_{G} \sum_{n \in \mathbb{N}_{s,A}}  \left\|  \int_{G_{s,A}}(E_{A+1}f)(y_{<s},x_s \cdots x_A,y_{\ge A+1} )\chi_{n^{(s)}}(y) \bar{\chi}_{n^{(s)}}(y_{<s},x_{\ge s}) d\mu(x_{s,A})  \right\|_{L_2(\cm)}^2 d\mu(y). \\
		\end{align*}
	Fix $y_0,\cdots ,y_{s-1},y_{A+1} \cdots$. We refer to the definition  of Fourier transform on $G_{s,A}$,
		\begin{align*}\int_{G_{s,A}}(E_{A+1}f)(y_{<s},x_s \cdots x_A,y_{\ge A+1} ) &\bar{\chi}_{n^{(s)}}(y_{<s},x_{\ge s}) d\mu(x_{s,A}) \\
		&=( (E_{A+1}f)(y_{<s},\cdots ,y_{\ge A+1}))^{\wedge} (n^{(s)}).
		\end{align*}
	Hence, by applying  the Hilbert-valued Plancherel theorem on $G_{s,A}$ and considering Equation~\eqref{2.4}, we arrive at
		\begin{flalign*}
				&\sum_{n \in \mathbb{N}_{s,A}}  \left\| M_s \int_{I_s(y)}f(y_{<s},x_{\ge s} )\chi_{n^{(s)}}(y) \bar{\chi}_{n^{(s)}}(y_{<s},x_{\ge s}) d\mu(x)  \right\|_{L_2(G;L_2(\cm))}^2 \\
			=&\int_{G_{<s}} \int_{G_{ \ge A+1}} \int_{G_{s,A}}\sum_{n \in \mathbb{N}_{s,A}}  \left\|  ( (E_{A+1}f)(y_{<s},\cdots ,y_{\ge A+1}))^{\wedge} (n^{(s)}) \right\|_{L_2(\cm)}^2 \vert \chi_{n^{(s)}}(y) \vert^2 d\mu(y) \\
			=&\int_{G_{<s}} \int_{G_{ \ge A+1}} \int_{G_{s,A}}  \left\|   (E_{A+1}f)(y_{<s},\cdots ,y_{\ge A+1}) \right\|_{L_2(G_{s,A};L_2(\mathcal{M}))}^2 \vert \chi_{n^{(s)}}(y) \vert^2 d\mu(y) \\
			=&\int_{G_{<s}} \int_{G_{ \ge A+1}}  \left\|   (E_{A+1}f)(y_{<s},\cdots ,y_{\ge A+1}) \right\|_{L_2(G_{s,A};L_2(\mathcal{M}))}^2 \Vert \chi_{n^{(s)}}(y) \Vert_{L_2(G_{s,A})}^2 d\mu(y) \\
			=& \left\|   E_{A+1}f \right\|_2^2 		\le \left\|f \right\|_2^2 . 
		\end{flalign*}
	Consequently,
		$$\left\|  {\sup_{\vert n \vert =A}}^+\mu^{-1}(J_s(y))\int_{J_s(y)}f(x)\chi_{n^{(s)}}(y) \bar{\chi}_{n^{(s)}}(x) d\mu(x)\right\|_2 \le  \Vert f \Vert_2.$$
\end{proof}

\section{Proof of Proposition \ref{1.2}}\label{section 5}
With Lemma~\ref{4.4} established in the preceding section,  we finish the proof of the main result of the paper---Theorem \ref{1.1}---by proving Proposition~\ref{1.2}. As mentioned in the introduction, this proposition is the most technical portion of the whole paper. 

Our argument necessitates two lemmas on the properties of the space $L_p(\cm;\ell_\infty)$, which are well-known to the experts. The first one is  a Fubini type result, see e.g. \cite{Hon13} for a proof. 
The second one follows from a similar argument as in Lemma~\ref{3.3}. 

\begin{lemma}\label{5.1}
	If  $(f_n)_n \in L_2(\mathcal N;\ell_\infty)$,  then 
	$$\Vert (f_n)_n\Vert_{L_2(\mathcal N;\ell_{\infty})}  \le  \left\|  \Vert (f_n)_n \Vert_{L_2(L_\infty(G_{<t>})\overline{\otimes} \cm;\ell_{\infty})} \right\|_{L_2({G_t})},$$
	where $G_{<t>}$ is defined in \eqref{Gt}.
\end{lemma}

\begin{lemma}\label{5.2}
	Consider a family of uniformly bounded functions  $g_n: G \rightarrow \mathbb{C}$  and a family of operator-valued function  $(f_n)_n$ in $L_2(\mathcal{N} ;\ell_\infty)$.  Then 
	$$\Vert (g_nf_n)_n \Vert_{L_2(\mathcal N;\ell_\infty)}  \le \sup_{ n }\Vert g_n \Vert_{\infty}  \Vert  (f_n)_n \Vert_{L_2(\mathcal N;\ell_\infty)}. $$
\end{lemma}
	

We now present the proof of Proposition~\ref{1.2}.

\begin{proof}[\bf{Proof of Proposition \ref{1.2}}]
	
	For any $y\in G$, the group $G$ can be partitioned as follows:
	$$G=J_0(y) =\left(  \bigcup_{t=0}^{A}J_t(y) \setminus J_{t+1}(y)  \right) \bigcup J_{A+1}(y).$$
	Applying the triangle inequality yields
	\begin{flalign*}
		\left\|  {\sup_{\vert n \vert =A}}^+ \int_{G} f(x)K_{n^{(s)},P_s}(y,x) d\mu(x)  \right\|_2  
		&\le \sum_{t=0}^{s-1}\left\|  {\sup_{\vert n \vert =A}}^+ \int_{J_t(y) \setminus J_{t+1}(y)}  f(x)K_{n^{(s)},P_s}(y,x) d\mu(x)  \right\|_2  \\
		&+ \sum_{t=s}^{A}\left\|  {\sup_{\vert n \vert =A}}^+
		\int_{J_t(y) \setminus J_{t+1}(y)} f(x)K_{n^{(s)},P_s}(y,x) d\mu(x)  \right\|_2 \\
		&+\left\|  {\sup_{\vert n \vert =A}}^+ \int_{J_{A+1}(y)} f(x)K_{n^{(s)},P_s}(y,x) d\mu(x)  \right\|_2  \\
		&=:\sum_{t=0}^{s-1}B_{1,t}+\sum_{t=s}^{A}B_{1,t}+B_0 . 
	\end{flalign*}
	Without loss of the generality, we may assume $f$ positive. Indeed, if $f$ is not positive, we can reduce the proof of the positive case by decomposing  $f=f_1-f_2+if_3-if_4$ where $f_k$ $(1 \le k\le 4)$  is positive  and $\Vert f_k \Vert_2 \le \Vert f \Vert_2$.

	We divide the rest of the  proof into three steps.
	
	\textbf{Step 1:}  We begin by examining the term $B_0$.  Note that for	 $x\in J_{A+1}(y),$ i.e. $x_i \in j_i(y_i)$ when $0 \le i \le A$ and $k \in \{n^{(s)},\cdots ,n^{(s)}+P_s-1 \}$, it follows that $r_\ell^i(x)=r_\ell^i(y)$ for $0 \le \ell \le A, 0 \le i \le k_\ell -1$. Hence,  Lemma \ref{4.1} allows us to get
	$$D_k(y,x) = \sum_{\ell=0}^{A}\mu^{-1}(J_\ell(y))\sum_{i=0}^{k_\ell-1}\vert r_\ell^i(y) \vert^2 \vert \chi_{k^{(\ell+1)}}(y) \vert^2 . $$
	It is observed that
	$$\sum_{i=0}^{k_\ell-1}\vert r_\ell^i(y) \vert^2 \le \sum_{i=0}^{p_\ell-1}\vert r_\ell^i(y) \vert^2 \le \sum_{i=0}^{p_\ell-1} (d_{\ell}^i)^2 =m_\ell \le \rho(G),$$
	where $d_\ell^i$ is the dimension of the representation corresponding to the character $r_\ell^i$.
	This leads us to estimate
	\begin{flalign*}
		B_0&=\left\| {\sup_{\vert n \vert =A}}^+ \int_{J_{A+1}(y)} f(x)K_{n^{(s)},P_s}(y,x) d\mu(x)  	\right\|_2  \\
		( \text{By Remark} \, \ref{2.3})&\leq \rho(G) \left\| {\sup_{\vert n \vert =A}}^+  \sum_{k=n^{(s)}}^{n^{(s)}+P_s-1} \sum_{\ell=0}^{A} \mu^{-1}(J_\ell(y))   \vert \chi_{k^{(\ell+1)}}(y) \vert^2   \int_{J_{A+1}(y)} f(x) d\mu(x)  	\right\|_2 \\
		( \text{By Lemma} \, \ref{4.3})&= \rho(G) \left\| {\sup_{\vert n \vert =A}}^+ \sum_{k=n^{(s)}}^{n^{(s)}+P_s-1} \sum_{\ell=0}^{A} \mu^{-1}(J_\ell(y))\mu(J_{A+1}(y))   \vert \chi_{k^{(\ell+1)}}(y) \vert^2    E_{A+1}f(y)   	\right\|_2  \\
		(\text{Triangle inequality} )&\leq \rho(G)\left\| {\sup_{\vert n \vert =A}}^+ \sum_{k=n^{(s)}}^{n^{(s)}+P_s-1} \sum_{\ell=0}^{s-1} \mu^{-1}(J_\ell(y))\mu(J_{A+1}(y))   \vert \chi_{k^{(\ell+1)}}(y) \vert^2    E_{A+1}f(y)   	\right\|_2   \\
		&\ \ \ +\rho(G)\left\| {\sup_{\vert n \vert =A}}^+ \sum_{k=n^{(s)}}^{n^{(s)}+P_s-1} \sum_{\ell=s}^{A} \mu^{-1}(J_\ell(y))\mu(J_{A+1}(y))   \vert \chi_{k^{(\ell+1)}}(y) \vert^2    E_{A+1}f(y)   	\right\|_2  \\
		&\eqqcolon \rho(G)(R_1+R_2) .
	\end{flalign*}
	We handle first the term $R_1$. Considering the parameters  $n$ and $k$ present  in the expression, we have that for any  $ j \ge s $, $k_j=n_j$, resulting in
	\begin{flalign}\label{e5.1}
		k^{(j)} = n^{(j)},\quad  \chi_{k^{(s)}}(y)=\chi_{n^{(s)}}(y).
	\end{flalign}
	Note that
	\begin{flalign*}
		&\ \ \ \sum_{k_{\ell+1}=0}^{p_{\ell+1}-1}\cdots \sum_{k_{s-1}=0}^{p_{s-1}-1}|\chi_{k^{(\ell+1)}}(y)|^2=
		\sum_{k_{\ell+1}=0}^{p_{\ell+1}-1}\cdots \sum_{k_{s-1}=0}^{p_{s-1}-1} \prod_{i=\ell+1}^{s-1} \vert r_i^{k_i}(y) \vert^2|\chi_{n^{(s)}}(y)|^2 \\
		&= \prod_{i=\ell+1}^{s-1}\sum_{k_i=0}^{p_i-1}\vert r_i^{k_i}(y) \vert^2 |\chi_{n^{(s)}}(y)|^2=\prod_{i=\ell+1}^{s-1} \mu_i^{-1}(j_i(y))|\chi_{n^{(s)}}(y)|^2,
	\end{flalign*}
	and 
    $$\mu^{-1}(J_\ell(y)) \times \mu_\ell^{-1}(j_\ell(y)) \times\prod_{i=\ell+1}^{s-1} \mu_i^{-1}(j_i(y)) = \mu^{-1}(J_s(y)).$$   
	Recall that $\mu_\ell(G_\ell) =1 $ for any $\ell \in \mathbb{N}$. Thus,
	\begin{flalign*}
			R_1&=\left\| {\sup_{\vert n \vert =A}}^+  \sum_{\ell=0}^{s-1}\mu^{-1}(J_\ell(y))\mu(J_{A+1}(y)) E_{A+1}f(y)\times  \r.\\
			&\ \ \ \ \ \ \times\lf.
		\sum_{k_0=0}^{p_0-1}\cdots\sum_{k_{\ell}=0}^{p_{\ell}-1} \left( \sum_{k_{\ell+1}=0}^{p_{\ell+1}-1}\cdots \sum_{k_{s-1}=0}^{p_{s-1}-1}  \prod_{j=\ell+1}^{s-1} \vert  r_j^{k_j}(y)\vert^2 \right)  \vert \chi_{n^{(s)}}(y) \vert^2 \right\|_2   \\
		&= \left\| {\sup_{\vert n \vert =A}}^+ \sum_{\ell=0}^{s-1}P_{\ell+1}  \mu(J_{A+1}(y)) E_{A+1}f(y) \mu_\ell(j_\ell(y)) \mu^{-1}(J_s(y)) \vert \chi_{n^{(s)}}(y) \vert^2 \right\|_2 \\
		&\leq  \left\| {\sup_{\vert n \vert =A}}^+ \sum_{\ell=0}^{s-1}P_{\ell+1} \mu(J_{A+1}(y)) E_{A+1}f(y) \mu^{-1}(J_s(y)) \vert \chi_{n^{(s)}}(y) \vert^2 \right\|_2 . \\
	\end{flalign*}
    On the other hand,
         \begin{flalign}\label{e5.2}
         	\vert \chi_{n^{(s)}}(y) \vert^2 &\le \sum_{n_s=0}^{p_s-1}\cdots \sum_{n_A=0}^{p_A-1} \vert \chi_{n^{(s)}}(y) \vert^2=\sum_{n_s=0}^{p_s-1}\cdots \sum_{n_A=0}^{p_A-1} \prod_{u=s}^{A} \vert r_u^{n_u}(y) \vert^2   \\
         	&=\prod_{u=s}^{A}\sum_{n_u=0}^{p_u-1}\vert r_u^{n_u}(y) \vert^2 =\prod_{u=s}^{A} \mu_u^{-1}(j_u(y)). \notag
         \end{flalign}
     According to the triangle inequality and  $p_k \ge 2$ for any  $k \in  \mathbb{N}$, we get
		\begin{flalign*}
		R_1&\leq  \left\| {\sup_{\vert n \vert =A}}^+ \sum_{\ell=0}^{s-1}P_{\ell+1} \mu^{-1}(J_s(y)) \mu(J_{A+1}(y)) E_{A+1}f(y) \prod_{u=s}^{A} \mu_u^{-1}(j_u(y))  \right\|_2  \\
		&\leq \sum_{\ell=0}^{s-1}P_{\ell+1} \left\| {\sup_{\vert n \vert =A}}^+   E_{A+1}f(y)  \right\|_2  = \sum_{\ell=1}^{s}P_{\ell}  \left\|  E_{A+1}f  \right\|_2  \\
		&= \sum_{\ell=1}^{s}\frac{P_s}{p_\ell \cdots p_{s-1}}  \left\|  E_{A+1}f  \right\|_2 \le  \sum_{\ell=1}^{s}  \frac{P_s}{2^{s-\ell}}  \left\|   E_{A+1}f \right\|_2  \\
		&=  \sum_{i=0}^{s-1}  \frac{P_s}{2^i}  \left\|   E_{A+1}f \right\|_2 \le  \sum_{i=0}^{\infty}  \frac{P_s}{2^i}  \left\|   E_{A+1}f \right\|_2 \le 2P_s\left\| f\right\|_2  .
	\end{flalign*}

	Next, we estimate the term $R_2$. By following \eqref{e5.1},
	\begin{flalign*}
		R_2&= \left\| {\sup_{\vert n \vert =A}}^+\sum_{k=n^{(s)}}^{n^{(s)}+P_s-1} \sum_{\ell=s}^{A} \mu^{-1}(J_\ell(y))\mu(J_{A+1}(y))   \vert \chi_{n^{(\ell+1)}}(y) \vert^2    E_{A+1}f(y)   	\right\|_2  \\
		&=  \left\| {\sup_{\vert n \vert =A}}^+  P_s\sum_{\ell=s}^{A} E_{A+1}f(y)\mu(J_{A+1}(y))\mu^{-1}(J_\ell(y))  \vert \chi_{n^{(\ell+1)}}(y)\vert^2   \right\|_2 . \\
	\end{flalign*}
Similar as \eqref{5.2}, we have
\begin{flalign*}
	\vert \chi_{n^{(\ell+1)}}(y)\vert^2 \le \prod_{u=\ell+1}^{A} \mu_u^{-1}(j_u(y)).
\end{flalign*}
	And by  Remark \ref{2.3}, we obtain	
	\begin{flalign*}
		R_2&\leq  \left\| {\sup_{\vert n \vert =A}}^+  P_s\sum_{\ell=s}^{A} E_{A+1}f(y)\mu(J_{A+1}(y))\mu_\ell(j_\ell(y))\mu^{-1}(J_{A+1}(y))     \right\|_2 \\
		&\leq  P_s(A-s+1) \left\| {\sup_{\vert n \vert =A}}^+  E_{A+1}f(y)     \right\|_2  \\
		&\leq  P_s(A-s+1) \left\|   f   \right\|_2.
	\end{flalign*}
	Hence,  we conclude
		$$B_0 \le \rho(G)(R_1+R_2) \leq \rho(G)(A-s+3)P_s \Vert f \Vert_2 . $$

	\textbf{Step 2:} We now deal with the term $\sum_{t=0}^{s-1}B_{1,t}$. 
	Consider $0\le t<s\leq A=|n|,k \in \{ n^{(s)},\cdots ,n^{(s)}+P_s-1\},x \in J_t(y) \setminus J_{t+1}(y) $, i.e. $x_t \notin j_t(y_t)$. Lemma~\ref{4.2} implies that for any $\ell \ge t+1 $,
	\begin{flalign}\label{e5.3}
		D_{P_\ell}(y,x)&= \mu^{-1}(J_\ell(y)) \boldsymbol{1}_{J_\ell(y)}(x)  \\
		&=\mu^{-1}(J_\ell(y)) \prod_{i=0}^{t-1}\boldsymbol{1}_{j_i(y)}(x_i) \times \prod_{i=t+1}^{\ell-1}\boldsymbol{1}_{j_i(y)}(x_i) \times \boldsymbol{1}_{j_t(y)}(x_t)  \notag  \\
		&=0.  \notag
	\end{flalign}
	Thus, Lemma~\ref{4.1} yields
	\begin{flalign*}
		D_k(y,x)&= \sum_{\ell =0}^{t-1} \mu^{-1}(J_\ell(y))\sum_{i=0}^{k_\ell -1}\vert r_\ell^i(y) \vert^2
		\prod_{j=\ell +1}^{t-1}\vert r_j^{k_j}(y) \vert^2 
		\chi_{k^{(t)}}(y) \bar{\chi}_{k^{(t)}}(x) \\
		&\ \ \ +\mu^{-1}(J_t(y)) \sum_{i=0}^{k_t-1}r_t^i(y) \bar{r}_t^i(x) \chi_{k^{(t+1)}}(y) \bar{\chi}_{k^{(t+1)}}(x) \\
		&=:\psi(y,x,k_0,\cdots,k_{t-1})\chi_{k^{(t)}}(y) \bar{\chi}_{k^{(t)}}(x)+B_{k,0}(y,x).
	\end{flalign*}
	Observing that 
	$$\sum_{k_t=0}^{p_t-1}
	r_t^{k_t}(y)\bar{r}_t^{k_t}(x)=\mu^{-1}(j_t(y))\boldsymbol{1}_{j_t(y)} (x)=0 . $$
	We then compute that
	\begin{flalign*}
		&\ \sum_{k=n^{(s)}}^{n^{(s)}+P_s-1}\psi(y,x,k_0,\cdots,k_{t-1})\chi_{k^{(t)}}(y) \bar{\chi}_{k^{(t)}}(x) \\
		&=\sum_{k_0=0}^{p_0-1}\cdots\sum_{k_{t-1}=0}^{p_{t-1}-1}\psi(y,x,k_0,\cdots,k_{t-1})\left( \sum_{k_t=0}^{p_t-1}
		r_t^{k_t}(y)\bar{r}_t^{k_t}(x)\right) \times \\
		&\ \ \ \ \ \ \times \sum_{k_{t+1}=0}^{p_{t+1}-1}\cdots\sum_{k_{s-1}=0}^{p_{s-1}-1}\chi_{k^{(t+1)}}(y) \bar{\chi}_{k^{(t+1)}}(x) \\
		&=0.
	\end{flalign*}
	Therefore, using the fact that $k^{(s)}= n^{(s)}$, we have 
	\begin{flalign*}
		K_{n^{(s)},P_s}(y,x)&=\sum_{k=n^{(s)}}^{n^{(s)}+P_s-1}B_{k,0}(y,x) \\
		&=\mu^{-1}(J_t(y)) \sum_{k_0=0}^{p_0-1}\cdots
		\sum_{k_{t-1}=0}^{p_{t-1}-1} \left( \sum_{k_t=0}^{p_t-1}\sum_{i=0}^{k_t-1} r_t^i(y)\bar{r}_t^i(x)  \prod_{j=t+1}^{s-1} \sum_{k_j=0}^{p_j-1}r_j^{k_j}(y) \bar{r}_j^{k_j}(x) \right) \times \\
		&\ \ \ \ \ \ \times \chi_{n^{(s)}}(y) \bar{\chi}_{n^{(s)}}(x) \\
		&=\mu^{-1}(J_t(y))P_t \sum_{i=0}^{p_t-2}(p_t-i-1)r_t^i(y)\bar{r}_t^i(x)
		\left( \prod_{u=t+1}^{s-1} \mu_u^{-1}(j_u(y_u))\boldsymbol{1}_{j_u(y_u)}(x_u)\right)   \times \\
		&\ \ \ \ \ \ \times \chi_{n^{(s)}}(y)\bar{\chi}_{n^{(s)}}(x). 
	\end{flalign*}
	Thus,  only  when $x_\ell \in j_\ell(y_\ell), \ell \in \{0,\cdots ,t-1,t+1,\cdots,s-1\}$,  term $K_{n^{(s)},P_s}$ does not vanish for $x \in J_t(y) \setminus J_{t+1}(y) $.  We denote	
	$$J_s(y_{<t>}) =j_0(y_0)\times \cdots \times j_{t-1}(y_{t-1}) \times j_{t+1}(y_{t+1}) \times \cdots \times   j_{s-1}(y_{s-1})\times \bigtimes_{i=s}^{\infty }G_i,$$ 
	and  recall
	$$G_{<t>}=\bigtimes_{i=0}^{t-1 }G_i \times \bigtimes_{i=t+1}^{\infty }G_i.$$
	Similarly, we denote  $x_{<t>} =(x_{< t}, x_{\ge t+1})$ and set $\widetilde{f}(x_t,x_{<t>}) := f(x)$. 
 	Note that $$\chi_{n^{(s)}}(y)=\chi_{n^{(s)}}(y_{<t>}),$$
	and $$\mu(J_s(y))= \mu(J_s(y_{<t>})) \times \mu_t(j_t(y_t)),$$
	which are natural for $t<s$. By the triangle inequality
	 \begin{flalign*}
	 		B_{1,t} &=P_t \left\| {\sup_{\vert n \vert =A}}^+ \sum_{i=0}^{p_t-2}(p_t-i-1) \mu^{-1}(J_s(y))\mu_t(j_t(y_t)) \times  \r.\\
	 	&\ \ \ \ \ \ \times\lf.
	 	\int_{J_s(y_{<t>}) \times (j_t(y_t))^c} r_t^i(y)\bar{r}_t^i(x) f(x) \chi_{n^{(s)}}(y) \bar{\chi}_{n^{(s)}} (x) d\mu(x)\right\|_2 \\
	 	&\leq P_t\sum_{i=0}^{p_t-2} (p_t-i-1)\left\|  {\sup_{\vert n \vert =A}}^+\mu^{-1}(J_s(y_{<t>}))
	 	\times  \r.\\
	 	&\ \ \ \ \ \ \times\lf. \int_{J_s(y_{<t>}) \times (j_t(y_t))^c} r_t^i(y)\bar{r}_t^i(x)f(x)\chi_{n^{(s)}}(y_{<t>}) \bar{\chi}_{n^{(s)}} (x_{<t>}) d\mu(x)
	 	\right\|_2  .\\
	 \end{flalign*}
 	Lemma 5.1 provides
      \begin{flalign*}
      	B_{1,t}&\le P_t \sum_{i=0}^{p_t-2} (p_t-i-1) \left\|    \left\|{\sup_{\vert n \vert =A}}^+\int_{(j_t(y_t))^c} \left(r_t^i(y)\bar{r}_t^i(x)\right)  \mu^{-1}(J_s(y_{<t>}))  	\times  \r.\r.\\
      	&\ \ \ \ \ \ \times\lf.\lf.
      	\int_{J_s(y_{<t>})} \widetilde{f}(x_t,x_{<t>})\chi_{n^{(s)}}(y_{<t>}) \bar{\chi}_{n^{(s)}} (x_{<t>})    
      	d\mu(x_{<t>})  
      	d\mu_t(x_t)  \right\|_{L_2(G_{<t>};L_2(\mathcal{M}))}	  \right\| _{L_2(G_t)}  \\
      	&\le P_t \sum_{i=0}^{p_t-2} (p_t-i-1) \left\|  \int_{(j_t{(y_t)})^c} \vert r_t^i(y_t)  \vert \vert \bar{r}_t^i(x)\vert   \left\|  {\sup_{\vert n \vert =A}}^+   \mu^{-1}(J_s(y_{<t>})) \times  \r.\r.\\
      	&\ \ \ \ \ \ \times\lf.\lf.
      	\int_{J_s(y_{<t>})} \widetilde{f}(x_t,x_{<t>})\chi_{n^{(s)}}(y_{<t>}) \bar{\chi}_{n^{(s)}} (x_{<t>})    
      	d\mu(x_{<t>})      \right\|_{L_2(G_{<t>};L_2(\mathcal{M}))}  d\mu_t(x_t) \right\|_{L_2(G_t)} .\\
      \end{flalign*}
  Through Lemma~\ref{4.4} and the Cauchy-Schwarz inequality, we obtain 
	 \begin{align*}
	 	 	B_{1,t} &\le P_t\sum_{i=0}^{p_t-2} (p_t-i-1) \left\| \int_{(j_t{(y_t)})^c}\vert r_t^i(y_t)  \vert \vert \bar{r}_t^i(x_t)\vert \left\| f(\cdots,x_t,\cdots) \right\|_{L_2(G_{<t>};L_2(\mathcal{M}))} d\mu_t(x_t)   \right\|_{L_2(G_t)}        \\
	 	 &\leq P_t\sum_{i=0}^{p_t-2} (p_t-i-1) \left\| \int_{G_t} \vert \bar{r}_t^i(x_t)\vert \left\| f(\cdots,x_t,\cdots) \right\|_{L_2(G_{<t>};L_2(\mathcal{M}))} d\mu_t(x_t)  \vert r_t^i(y_t)  \vert \right\|_{L_2(G_t)}  \\
	 	 &\leq P_t \sum_{i=0}^{p_t-2} (p_t-i-1) \left\|   \left\| f \right\|_2     \Vert \bar{r}_t^i(x_t) \Vert_{L_2(G_t)} \vert r_t^i(y_t)  \vert \right\|_{L_2(G_t)} \\
	 	 &= P_t \sum_{i=0}^{p_t-2} (p_t-i-1)  \left\| f \right\|_2 \left\| r_t^i(y_t) \right\|_{L_2(G_t)}= P_t \sum_{i=0}^{p_t-2} (p_t-i-1)  \left\| f \right\|_2 . 
	 \end{align*}	
	Since 
$$\sum_{i=0}^{p_t-2} (p_t-i-1) =\frac{p_t(p_t-1)}{2} \le \frac{(p_t)^2}{2} \le \frac{(m_t)^2}{2} \le \rho(G)^2,$$
we conclude
	$$B_{1,t} \le P_t\rho(G)^2 \Vert f \Vert_2.$$
	
	\textbf{Step 3:} Finally, we turn our attention to the term  $\sum_{t=s}^{A}B_{1,t}$. In this case, we have  $A \ge t \ge s$. Observe that $$x \in J_t(y) \setminus J_{t+1}(y), k \in \{ n^{(s)},\cdots , n^{(s)}+P_s-1\}.$$ 
	By \eqref{e5.3} and Lemma \ref{4.1}, we have
	\begin{align*}
		D_k(y,x)&= \sum_{\ell =0}^{t-1} \mu^{-1}(J_\ell(y))\sum_{i=0}^{k_\ell -1}\vert r_\ell^i(y) \vert^2
		\prod_{j=\ell +1}^{t-1}\vert r_j^{k_j}(y) \vert^2 
		\chi_{k^{(t)}}(y) \bar{\chi}_{k^{(t)}}(x) \\
		&\ \ \ +\mu^{-1}(J_t(y)) \sum_{i=0}^{k_t-1}r_t^i(y) \bar{r}_t^i(x) \chi_{k^{(t+1)}}(y) \bar{\chi}_{k^{(t+1)}}(x) \\
		&=:B_{k,1}(y,x)+B_{k,0}(y,x).
	\end{align*}
	This leads to
	   \begin{align*}
			&\left\| {\sup_{\vert n \vert =A}}^+  \int_{J_t(y) \setminus J_{t+1}(y)} f(x) K_{n^{(s)},P_s}(y,x) d\mu(x) \right\|_2 \\
			\leq& \left\| {\sup_{\vert n \vert =A}}^+ \sum_{k=n^{(s)}}^{n^{(s)}+P_s-1} \int_{J_t(y) \setminus J_{t+1}(y)} f(x)B_{k,1}(y,x)d\mu(x) \right\|_2 \\
			&+\left\| {\sup_{\vert n \vert =A}}^+ \sum_{k=n^{(s)}}^{n^{(s)}+P_s-1} \int_{J_t(y) \setminus J_{t+1}(y)} f(x)B_{k,0}(y,x)d\mu(x) \right\|_2 \\
			=:&N_1+N_2.  
	   \end{align*}
	
	We now estimate  $N_1$. According to Lemma $\ref{4.2}$ and Equation~\eqref{e5.1}, we obtain
	\begin{flalign*}
		&\ \ \ \ \sum_{k_0=0}^{p_0-1}\cdots \sum_{k_{s-1}=0}^{p_{s-1}-1}\sum_{\ell = 0 }^{t-1} \mu^{-1}(J_\ell(y))\prod_{j=\ell+1}^{t-1}\vert r_j^{k_j}(y) \vert^2 \\
		&=\sum_{\ell=0}^{s-2} \mu^{-1}(J_{\ell}(y))\sum_{k_0=0}^{p_0-1}\cdots\sum_{k_\ell=0}^{p_\ell-1}  \left( \sum_{k_{\ell+1}=0}^{p_{\ell+1}-1} \cdots \sum_{k_{s-1}=0}^{p_{s-1}-1}\prod_{j=\ell+1}^{s-1}\vert r_j^{k_j}(y) \vert^2 \right) \prod_{j=s}^{t-1} \vert r_j^{n_j}(y) \vert^2  \\
		&\ \ \ \ \ 
		+\sum_{k_0=0}^{p_0-1}\cdots \sum_{k_{s-1}=0}^{p_{s-1}-1}\sum_{\ell = s-1 }^{t-1} \mu^{-1}(J_\ell(y))\prod_{j=\ell+1}^{t-1}\vert r_j^{n_j}(y) \vert^2   \\
		&\leq \sum_{\ell=0}^{s-2} P_{\ell+1} \mu_\ell (j_\ell (y_\ell)) \mu^{-1}(J_s(y)) \left( \sum_{n_s=0}^{p_s-1}\cdots \sum_{n_{t-1}=0}^{p_{t-1}-1}\prod_{j=s}^{t-1}\vert r_j^{n_j}(y) \vert^2 \right)   \\
		&\ \ \ \ \  +P_s\sum_{\ell=s-1}^{t-1}\mu^{-1}(J_\ell(y)) \left( \sum_{n_{\ell+1}=0}^{p_{\ell+1}-1}\cdots \sum_{n_{t-1}=0}^{p_{t-1}-1}\prod_{j=\ell+1}^{t-1}\vert r_j^{n_j}(y) \vert^2 \right)   \\
		&\leq \sum_{\ell=0}^{s-2} \mu_\ell (j_\ell (y_\ell)) P_{\ell+1} \mu^{-1}(J_t(y)) + P_s\sum_{\ell=s-1}^{t-1} \mu_\ell (j_\ell (y_\ell)) \mu^{-1}(J_t(y))\\
		&=\left(  \sum_{\ell = 0}^{s-2} P_{\ell+1}  + \sum_{\ell=s-1}^{t-1}P_s \right) \mu^{-1}(J_t(y))  \\
		&=\left(  \sum_{\ell=1}^{s-1}P_{\ell}  + (t-s+1)P_s \right) \mu^{-1}(J_t(y))= \left( \sum_{\ell=1}^{s-1}\frac{P_s}{p_\ell \cdots p_{s-1}}   + (t-s+1)P_s \right) \mu^{-1}(J_t(y))  \\
		&\le \left( \sum_{\ell=1}^{s-1} \frac{P_s}{2^{s-\ell}}    + (t-s+1)P_s \right) \mu^{-1}(J_t(y))= \left( \sum_{i=1}^{s-1} \frac{P_s}{2^{i}}    + (t-s+1)P_s \right) \mu^{-1}(J_t(y))  \\
		&\le \left( \sum_{i=1}^{\infty} \frac{P_s}{2^{i}}    + (t-s+1)P_s \right) \mu^{-1}(J_t(y)) = (t-s+2)  P_s\mu^{-1}(J_t(y)) .
	\end{flalign*}
	Applying this estimation, Lemma \ref{5.2} and Equation~\eqref{e5.1},   allows us to calculate
	\begin{flalign*}
		N_1 &= \left\| {\sup_{\vert n \vert =A}}^+ \sum_{k_0=0}^{p_0-1}\cdots  \sum_{k_{s-1}=0}^{p_{s-1}-1}\sum_{\ell = 0 }^{t-1} \mu^{-1}(J_\ell(y))\sum_{i=0}^{k_\ell-1}\vert r_\ell^i(y) \vert^2 \prod_{j=\ell+1}^{t-1}\vert r_j^{k_j}(y) \vert^2 \times  \r. \notag\\
		&\ \ \ \ \ \ \ \times\lf.
		\int_{J_t(y) \setminus J_{t+1}(y)} f(x)\chi_{n^{(t)}}(y) \bar{\chi}_{n^{(t)}}(x)d\mu(x) \right\|_2 \notag\\
		&\leq  \sup_{ \vert n \vert=A }  \left\|  \sum_{k_0=0}^{p_0-1}\cdots \sum_{k_{s-1}=0}^{p_{s-1}-1}\sum_{\ell = 0 }^{t-1} \mu^{-1}(J_\ell(y))\sum_{i=0}^{k_\ell-1}\vert r_\ell^i(y) \vert^2 \prod_{j=\ell+1}^{t-1}\vert r_j^{k_j}(y) \vert^2 \mu(J_t(y)) \right\|_{\infty} \times \\
		&\ \ \ \ \ \ \  \times \left\| {\sup_{\vert n \vert =A}}^+ \mu^{-1}(J_t(y))  \int_{J_t(y) \setminus J_{t+1}(y)} f(x)\chi_{n^{(t)}}(y) \bar{\chi}_{n^{(t)}}(x)d\mu(x) \right\|_2    \notag\\
		&\leq (t-s+2) P_s\left\| {\sup_{\vert n \vert =A}}^+ \mu^{-1}(J_t(y))  \int_{J_t(y) \setminus J_{t+1}(y)} f(x)\chi_{n^{(t)}}(y) \bar{\chi}_{n^{(t)}}(x)d\mu(x) \right\|_2  \notag\\
		&\leq (t-s+2) P_s   \left\| {\sup_{\vert n \vert =A}}^+ \mu^{-1}(J_t(y))  \int_{J_t(y)} f(x)\chi_{n^{(t)}}(y) \bar{\chi}_{n^{(t)}}(x)d\mu(x) \right\|_2  \\
		&\ \ \ \ \ \ \ + (t-s+2) P_s\left\| {\sup_{\vert n \vert =A}}^+ \mu^{-1}(J_t(y))  \int_{J_{t+1}(y)} f(x)\chi_{n^{(t)}}(y) \bar{\chi}_{n^{(t)}}(x)d\mu(x) \right\|_2  	\notag\\
		&=:(t-s+2)P_s( E + F) \notag.
	\end{flalign*}
	It follows from Lemma \ref{4.4} that
     $$E \le \Vert f \Vert_2.$$ 
	 Regarding  the term $F$, by applying the method of estimating $B_{1,t}$ in \textbf{Step 2} and using Lemma \ref{5.1}, we get that
    \begin{flalign*}
    	F&=\left\| {\sup_{\vert n \vert =A}}^+ \mu^{-1}(J_t(y))  \int_{J_t(y_{<t>})\times j_t(y_t)} r_t^{n_t}(y_t)\bar{r}_t^{n_t}(x_t) f(x)\chi_{n^{(t)}}(y_{<t>}) \bar{\chi}_{n^{(t)}}(x_{<t>})d\mu(x) \right\|_2 \\
    	&\le  \left\|   \left\|   {\sup_{\vert n \vert =A}}^+  \mu^{-1} (J_{t}(y_{<t>}))  \int_{j_t(y_t)}   r_t^{n_t}(y_t)  \bar{r}_t^{n_t}(x_t)   
    	\times  \r.\r.\\
    	&\ \ \ \ \ \ \times\lf. \lf.
    	\int_{J_{t}(y_{<t>}) } \widetilde{f}(x_t,x_{<t>})\chi_{n^{(t)}}(y_{<t>}) \bar{\chi}_{n^{(t)}}(x_{<t>})  d\mu(x) \right\|_{L_2(G_{<t>};L_2(\mathcal{M}))}  \right\|_{L_2(G_t)} \\
    	&\le \left\|  \int_{j_t(y_t)}\vert r_t^{n_t}(y_t) \vert \vert  \bar{r}_t^{n_t}(x_t) \vert      \left\|   {\sup_{\vert n \vert =A}}^+  \mu^{-1} (J_{t}(y_{<t>}))    
    	\times  \r.\r.\\
    	&\ \ \ \ \ \ \times\lf. \lf.
    	\int_{J_{t}(y_{<t>}) } \widetilde{f}(x_t,x_{<t>})\chi_{n^{(t)}}(y_{<t>}) \bar{\chi}_{n^{(t)}}(x_{<t>})  d\mu(x_{<t>}) \right\|_{L_2(G_{<t>};L_2(\mathcal{M}))} d\mu_t(x_t)  \right\|_{L_2(G_t)} \\
    	&\le \left\|     \int_{j_t(y_t)}   \vert  \bar{r}_t^{n_t}(x_t) \vert   \left\| f(\cdots,x_t,\cdots)   \right\|_{L_2(G_{<t>};L_2(\mathcal{M}))}       d\mu_t(x_t) \vert r_t^{n_t}(y_t) \vert \right\|_{L_2(G_t)}   \\
    	&\le \left\| \Vert f \Vert_2  \Vert \bar{r}_t^{n_t}(x_t)  \Vert_{L_2(G_t)}   r_t^{n_t}(y_t)\right\|_{L_2(G_t)}= \Vert f \Vert_2.
    \end{flalign*}
	Therefore,
	  $$N_1 \le 2(t-s+2) P_s \Vert f \Vert_2.$$

	 Next we address  $N_2$.  Recall that for $t \ge s$,
	$$\chi_{k^{(t+1)}}(y) \bar{\chi}_{k^{(t+1)}}(x) = \chi_{n^{(t+1)}}(y) \bar{\chi}_{n^{(t+1)}}(x).$$
	Mirroring the approach for  $B_{1,t}$ in \textbf{Step 2}, we deduce that
	\begin{flalign*}
		N_2&=\left\| {\sup_{\vert n \vert =A}}^+ \sum_{k=n^{(s)}}^{n^{(s)}+P_s-1} \mu^{-1}(J_{t+1}(y_{<t>}))  \times  \r.\\
		&\ \ \ \ \ \ \times\lf.
		\sum_{i=0}^{k_t-1} \int_{J_{t+1}(y_{<t>}) \times (j_t(y))^c}r_t^i(y)\bar{r}_t^i(x) f(x) \chi_{n^{(t+1)}}(y) \bar{\chi}_{n^{(t+1)}}(x) d\mu(x)  \right\|_2  \\
		&\leq P_s \sum_{i=0}^{k_t-1}\left\| {\sup_{\vert n \vert =A}}^+ \mu^{-1} (J_{t+1}(y_{<t>})) \int_{J_{t+1}(y_{<t>}) \times (j_t(y))^c}r_t^i(y)\bar{r}_t^i(x) f(x) \chi_{n^{(t+1)}}(y) \bar{\chi}_{n^{(t+1)}}(x) d\mu(x)   \right\|_2   \\
		&= P_s \sum_{i=0}^{k_t-1}\left\| {\sup_{\vert n \vert =A}}^+ \mu^{-1} (J_{t+1}(y_{<t>})) \times  \r.\\
		&\ \ \ \ \ \ \times\lf. \int_{J_{t+1}(y_{<t>}) \times (j_t(y))^c}r_t^i(y_t)\bar{r}_t^i(x_t) \widetilde{f}(x_t,x_{<t>}) \chi_{n^{(t+1)}}(y_{<t>}) \bar{\chi}_{n^{(t+1)}}(x_{<t>}) d\mu(x)   \right\|_2 \\
		&\le P_s \sum_{i=0}^{k_t-1}\left\|   \left\|   {\sup_{\vert n \vert =A}}^+   \int_{(j_t(y))^c}   r_t^i(y_t)   \bar{r}_t^i(x_t)     \mu^{-1} (J_{t+1}(y_{<t>})) 
		\times  \r.\r.\\
		&\ \ \ \ \ \ \times\lf. \lf.
		 \int_{J_{t+1}(y_{<t>}) } \widetilde{f}(x_t,x_{<t>}) \chi_{n^{(t+1)}}(y_{<t>}) \bar{\chi}_{n^{(t+1)}}(x_{<t>})  d\mu(x) \right\|_{L_2(G_{<t>};L_2(\mathcal{M}))}   \right\|_{L_2(G_t)} \\
		&\leq P_s \sum_{i=0}^{k_t-1}   \left\|     \int_{(j_t(y))^c}  \vert r_t^i(y_t) \vert \vert  \bar{r}_t^i(x_t) \vert   \left\|   {\sup_{\vert n \vert =A}}^+   \mu^{-1} (J_{t+1}(y_{<t>})) 
		\times  \r.\r.\\
		&\ \ \ \ \ \ \times\lf.\lf. 
		 \int_{J_{t+1}(y_{<t>}) } \widetilde{f}(x_t,x_{<t>}) \chi_{n^{(t+1)}}(y_{<t>}) \bar{\chi}_{n^{(t+1)}}(x_{<t>})    d\mu(x_{<t>}) \right\|_{L_2(G_{<t>};L_2(\mathcal{M}))} 
		d\mu_t(x_t)  \right\|_{L_2(G_t)}  \\
		&\le P_s \sum_{i=0}^{k_t-1}   \left\|     \int_{(j_t(y))^c}   \vert  \bar{r}_t^i(x_t) \vert   \left\| f(\cdots,x_t,\cdots)   \right\|_{L_2(G_{<t>};L_2(\mathcal{M}))}        d\mu_t(x_t) \vert r_t^i(y_t) \vert \right\|_{L_2(G_t)}   \\
		&\leq  P_s \sum_{i=0}^{k_t-1} \left\| \Vert f \Vert_2 
		 \Vert \bar{r}_t^i(x_t)  \Vert_{L_2(G_t)}    r_t^i(y_t)\right\|_{L_2(G_t)}  \\
		&\leq p_t P_s\Vert f \Vert_2 \le m_tP_s\Vert f \Vert_2 \le \rho(G) P_s\Vert f \Vert_2.
	\end{flalign*}
	Thus, we have
	$$B_{1,t} =N_1+N_2 \le (2(t-s+2)+\rho(G))P_s \Vert f \Vert_2.$$

Finally, validating that $(A-s) \le (A-s)^2,   2 \le \rho(G) \le \rho(G)^2$  enables us to compute
\begin{flalign*}
	\sum_{t=0}^{s-1}P_t &= \sum_{t=0}^{s-1}  \frac{P_s}{p_{t} \cdots p_{s-1}  }  \\
	&\le \sum_{t=0}^{s-1}\frac{P_s}{2^{s-t}} = \sum_{i=1}^{s}\frac{P_s}{2^{i}}  \\
	&\le \sum_{i=1}^{\infty}\frac{P_s}{2^{i}}  = P_s.
\end{flalign*}
We conclude that
	\begin{flalign*}
	&\left\|  {\sup_{\vert n \vert =A}}^+ \int_{G} f(x)K_{n^{(s)},P_s}(y,x) d\mu(x)  \right\|_2 \\
	&\le \sum_{t=0}^{s-1}\rho(G)^2P_t\Vert f \Vert_2 +\sum_{t=s}^{A}(2(t-s+2)+\rho(G))P_s \Vert f \Vert_2 
	+\rho(G)(A-s+3)P_s \Vert f \Vert_2 \\
	&\le \rho(G)^2 \left(  \sum_{t=0}^{s-1}P_t +\sum_{t=s}^{A}2(t-s+3)P_s 
	+(A-s+3)P_s \right)  \Vert f \Vert_2 \\
	&\leq \rho(G)^2\left( 1+2\times \frac{(A-s+1)(A-s+6)}{2}+ (A-s+3) \right)  P_s\Vert f \Vert_2 \\
	&\leq \rho(G)^2 \left( (A-s)^2+7(A-s)+6 + (A-s)+4 \right)  P_s\Vert f \Vert_2 \\
	&\leq \rho(G)^2 \left( 9(A-s)^2 +10 \right)  P_s\Vert f \Vert_2 \\
	&\le 10\rho(G)^2 ((A-s)^2+1)P_s \Vert f \Vert_2.
	\end{flalign*}
	Hence, we complete the proof of  Proposition \ref{1.2}.
\end{proof}

\subsubsection*{Acknowledgment}
	All authors were supported by the Natural Science Foundation of China (No.12031004). F. Ding and G. Hong were supported by the National Natural Science Foundation of China (No. 12071355, No. 12325105). X. Wang was partially supported by National Research Foundation of Korea (NRF) Grant NRF-2022R1A2C1092320, NRF Grant No.2020R1C1C1A01009681.

\end{document}